\documentclass[10pt]{article}
\usepackage[utf8]{inputenc}

\usepackage{parskip}

\usepackage{geometry}[margin = 0.5 in]
\usepackage{amsthm}
\usepackage{hyperref}
\usepackage{caption}
\usepackage{float}
\usepackage{changepage}
\usepackage{amssymb}
\usepackage{amsmath}
\usepackage{mathtools}
\usepackage[backend=biber]{biblatex}
\usepackage{microtype}
\usepackage{graphicx}
\usepackage{multirow}

\DeclareMathOperator{\mct}{mct}
\DeclareMathOperator{\mcg}{mcg}
\usepackage[many]{tcolorbox}
\usepackage{tikz}
\usepackage{listings}
\usepackage{array}
\newcolumntype{P}[1]{>{\centering\arraybackslash}p{#1}}

\usetikzlibrary {math}

\newcommand{\equic}[2][1 cm]{
  \foreach \i in {1,...,#2} {
    \coordinate (N\i) at (\i*360/#2:#1);
    \fill[black] (N\i) circle (.05 cm) node[anchor=west]{$\i\cdot 5$};
  }
  \draw (N#2) -- (N1);
  \foreach \i in {2,...,#2} {
    \pgfmathparse{\i-1}
    \edef\j{\pgfmathresult}
    \draw (N\i) -- (N\j);
  }
}

\definecolor{calpolypomonagreen}{rgb}{0.12, 0.3, 0.17}

\title{Integer Complexity Generalizations in Various Rings}
\author{Aarya Kumar, Siyu Peng, Vincent Tran}
\begin{document}

\parindent=0pt
\theoremstyle{definition}
\newtheorem{theorem}{Theorem}[section]

\tcolorboxenvironment{theorem}{
  enhanced,
  borderline={0.8pt}{0pt}{blue},
  borderline={0.4pt}{2pt}{blue},
  boxrule=0.4pt,
  colback=white,
  coltitle=black,
}

\newtheorem{corollary}{Corollary}[theorem]
\newtheorem{example}[theorem]{Example}
\newtheorem{lemma}[theorem]{Lemma}
\newtheorem{conj}[theorem]{Conjecture}

\tcolorboxenvironment{lemma}{
  enhanced,
  borderline={0.8pt}{0pt}{violet},
  borderline={0.4pt}{2pt}{violet},
  boxrule=0.4pt,
  colback=white,
  coltitle=black,
}

\tcolorboxenvironment{conj}{
  enhanced,
  borderline={0.8pt}{0pt}{red},
  borderline={0.4pt}{2pt}{red},
  boxrule=0.4pt,
  colback=white,
  coltitle=black,
}

\newtheorem*{fact}{Fact}
\newtheorem*{obs}{Observation}
\newtheorem*{claim}{Claim}
\theoremstyle{remark}
\newtheorem*{remark}{Remark}

\theoremstyle{definition}
\newtheorem*{definition}{Definition}

\tcolorboxenvironment{definition}{%
  colback=gray!25, colframe=gray!25,
  coltitle=black, fonttitle=\bfseries,
  sharp corners,
  width=\linewidth}

\newtheorem{prop}[theorem]{Proposition}
\tcolorboxenvironment{prop}{
  enhanced,
  borderline={0.8pt}{0pt}{purple},
  borderline={0.4pt}{2pt}{purple},
  boxrule=0.4pt,
  colback=white,
  coltitle=black,
}

\tcolorboxenvironment{remark}{
  borderline={1pt}{4pt}{black},
  boxrule=0.8pt,
  colback=white,
  coltitle=black,
  sharp corners
}

\tcolorboxenvironment{obs}{
  borderline={1pt}{4pt}{black},
  colback=white,
  coltitle=black,
  sharp corners,
  boxrule=0.4pt
}

\newcommand{\N}{\mathbb{N}}
\newcommand{\Ima}{\text{Im}}


\maketitle

\begin{abstract}
    In this paper, we investigate generalizations of the Mahler-Popkens complexity of integers. Specifically, we generalize to $k$-th roots of unity, polynomials over the naturals, and the integers mod $m$. In cyclotomic rings, we establish upper and lower bounds for integer complexity, investigate the complexity of roots of unity using cyclotomic polynomials, and introduce a concept of ``minimality.'' In polynomials over the naturals, we establish bounds on the sizes of complexity classes and establish a trivial but useful upper bound. In the integers mod $m$, we introduce the concepts of ``inefficiency'', ``resilience'', and ``modified complexity.'' In hopes of improving the upper bound on the complexity of the most complex element mod $m$, we also use graphs to visualize complexity in these finite rings.
\end{abstract}

\section{Introduction}
Integer complexity is a system that attempts to quantify our intuitions on how `complex' a number is. 

\begin{definition}
    The \textbf{complexity} of a number $n$, denoted as $||n||$ is equal to the smallest number of 1's needed to write $n$ as an equation using addition and multiplication.
\end{definition}

For example, 
\[
    6 = (1+1)(1+1+1)
\]
which implies $||6|| \leq 5$. In fact, this is the complexity of 6. Here is another example for 7. 
\[
    7 = (1+1)(1+1+1)+1
\]
Though the following definitions are somewhat pedantic, they are necessary to reduce ambiguity. We take the definition of expression presented in \cite{Reyna2021}, where the idea is defined more rigorously. 
\begin{definition}
    An \textbf{expression} refers to a sequence of $x$'s, which are joined together by $+, \cdot,$ and $()$ in the usual sense.
\end{definition}
For example, a valid expression is
    \[(x+x+x)\cdot (x+x) + x \cdot x \cdot x\]
As shorthand, we would write 
    \[3x \cdot 2x + x^3 \]
for the above. Keep in mind that this is merely shorthand and not a valid expression. 
\begin{definition}
    The \textbf{representation} of a \textbf{base} $k$ is an expression with $x=\zeta_k$ which evaluates to be $n$.
\end{definition}
Note that the examples that we've provided have $x=1$ and $k=1$. 

\begin{definition}
    A representation is \textbf{optimal} when $||n||$ is equal to the number of $x$'s used in the expression. 
\end{definition}

 There is some pre-existing literature on the subject, but none has yet explored the potential of integer complexity in other rings. In this paper, we will begin by generalizing the idea of integer complexity into cyclotomic fields. Then, we will investigate integer complexity in polynomial rings, an application that arises naturally from the definition of an expression. Finally, we shall look into complexity in the integers mod $m$ and the integers mod $p$, where $p$ is prime. 
 
 Much of this exploration shall consist of establishing upper and lower bounds on complexity and the size of complexity classes, often by means of Horner's method. When investigating complexity in cyclotomic rings, we will additionally focus on finding some integers $k$ for which $||1||_k = k$ in the $k$th cyclotomic field, and some integers $k$ for which this is not the case. When dealing with finite (quotients of integers) rings and fields, we will explore the idea that, by well-ordering, there exists an element(s) with the largest complexity, and that the complexity of elements changes depending on what unit is used as a base in place of 1.

\section{Cyclotomic Fields}
As one studies the complexity of some integers, they will sometimes wish that they could subtract. For example, Mersenne primes would generally have a much smaller complexity if one could subtract 1. The logical next step in integer complexity would be to use roots of unity. For instance when we let $k=\zeta_2=-1$, we can now subtract and add. However, note that $||1||_2 = k$ as we change $k$. Through numerical calculation, we have observed that most numbers increase in complexity while others decrease. 

\begin{definition}
    While we are working with cyclotomic fields, let $\zeta_k = e^{2\pi i/k}$. The complexity of $n$ with base $\zeta_k$ is denoted as $||n||_k$. 
\end{definition}

We are familiar with $\zeta_4$, also known as $i$. Let us consider integer complexity when we use $i$ as the base. We have easy access to subtraction with $i^2$. However in other $k$, subtracting may not be so easy, but is still doable.

\begin{figure}[H]
    \centering
    \captionsetup{justification=centering,margin=1cm}
    \includegraphics[width = .9\textwidth]{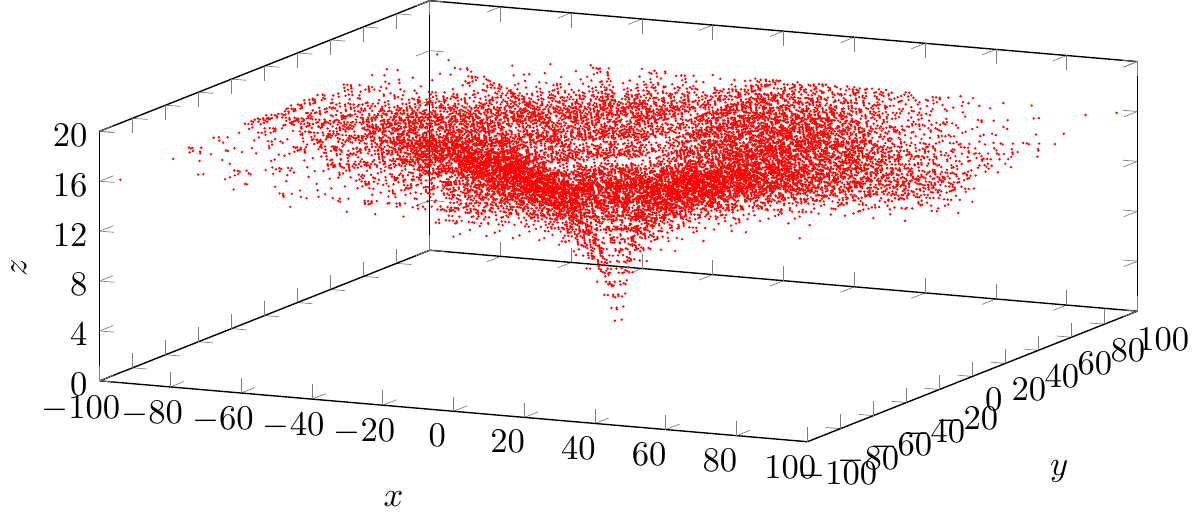}
    \caption{A plot of the real component ($x$) and imaginary component ($y$) vs. the complexity of that complex number (on the $z$ axis) using $\zeta_4$.}
    \label{fig:k43d}
\end{figure}

\begin{prop}
We can construct a representation of all elements of $\mathbb{Z}[\zeta_k]$ with $\zeta_k$. 
\end{prop}
\begin{proof}
	We know we can represent any element $a \in \mathbb{N}[\zeta_k]$ as \(a_0 + a_1\zeta + a_2\zeta^2 + \cdots  \). For some arbitrary positive integer $j$, we can say
	\[
	-\zeta_k^j = \sum_{\substack{\ell \in \mathbb{Z}\\ 0 \leq \ell < k \\ \ell \neq j}} ^k \zeta_k^\ell
	\]
	Thus we can obtain any negative coefficients for each root of unity through the repeated addition of this above representation. We can represent 0 by summing all the roots of unities. 
\end{proof}

We can replace the $1$'s in the $k=1$ representation of a natural number with $\zeta_k^k$ and find that $\forall n \in \mathbb{N}, ||n||_k \leq k||n||$. In most cases, this is not the optimal representation.
To create bounds for $k \geq 1$, we take inspiration from pre-existing bounds of integer complexity when $k=1$ and attempt to generalize them for higher $k$. 

\begin{lemma} \label{lem:magnitude}
    Let $f(x)$ be an arbitrary expression that uses the input as the symbol. Then, for all $f,  |f(\zeta_k)| \leq f(1)$. 
\end{lemma}
\begin{proof}
    Let $f(x), g(x)$ be an arbitrary expression that uses the input as the symbol. For example, if $f(x) = x(x+x)$, then $f(1) = 1(1+1)$ and $f(i) = i(i+i)$. We prove by induction. As our inductive step, consider $f(1), |f(\zeta_k)|$, and $|g(\zeta_k)|$. If $|f(\zeta_k)| \leq f(1)$ and $|g(\zeta_k)| \leq g(1)$, Then, we can see that 
    \[
        |f(\zeta_k) + g(\zeta_k)| \leq f(1) + g(1) 
    \]
    which we know is true, considering the geometry of vector addition in the complex plane/the triangle inequality, and note that
    \[
        |f(\zeta_k)g(\zeta_k)| \leq f(1)g(1)
    \]
    which is also true since magnitude is multiplicative.
    
    Then for our base case, we can check that for $f(x) = g(x) = x$, that $|f(\zeta_k)| = 1 = f(1)$ and $|f(\zeta_k) + g(\zeta_k)| = |2\zeta_k|$, which satisfies our hypotheses. We can build all other expressions from these, by definition. Thus, by induction, we can say that for all functions $f$, 
    \[
        |f(\zeta_k)| \leq f(1)
    \]
\end{proof}

\begin{theorem}
    \[
    3 \log_3 |n|  \leq ||n||_k
    \]
\end{theorem}
\begin{proof}
    We know that the lower bound is true for when $k=1$ by Altman, \cite{Altman2012}. This is obtained from the representation
    \[
        n = (1+1+1)(1+1+1) \cdots
    \]
    having the largest value for its complexity.
    
    Let $g(x) = (x+x+x)(x+x+x) \cdots $. Since we know that $g(1)$ has the largest magnitude for its complexity, then we can say that any other expression $f(x)$ with the same complexity as $g(x)$ satisfies that $g(1) \geq f(1)$.
    
    We can see that $|g(\zeta_k)| = g(1), \forall k \in \mathbb{N}.$ By Lemma \ref{lem:magnitude}, for all expressions $f$, we have $|f(\zeta_k)| \leq f(1).$ But by Altman, $f(1) \leq g(1) = |g(\zeta_k)|$, so by transitivity, we can say that $|f(\zeta_k)| \leq |g(\zeta_k)|$. Then, for $n$ with complexity $k$, we have $|n| \leq 3^{\frac{k}{3}}$. This implies that $3 \log_3 |n|  \leq ||n||_k$. Thus, we obtain our lower bound.
\end{proof}

\begin{remark}
    Note that as in the $k=1$ case, this lower bound is the best lower log bound possible, as the lower bound is obtained by $g(\zeta_k) = 3^k = g(1)$. Also note that it is obtained infinitely often, by $3^{ki} \,\forall i \in\N$.
\end{remark}

\begin{prop} \label{prop:cycloweakupper}
\[
||n||_k \leq \frac{k(k+5)}{2} \cdot \left\lceil \frac{\lceil \log_2 n \rceil + 1}{k} \right\rceil
\]
\end{prop}
\begin{proof}
Again, we take inspiration from the proof of the upper bound by Altman, \cite{Altman2012}, which involves representing a natural number $n$ in binary. We use Horner's method. For an example of an application of this method, see Altman, \cite{Altman2012}. Then, we can see that the `worst case scenario' would arise in the case when $n$ in binary is represented as a string of all 1's. We can represent that in $k=1$ as
\[
    (((1+1)+1)(1+1) + 1) (1+1) + \ldots 
\]
It is tempting to generalize this to all $k$'s by replacing the 1's with $\zeta_k$. However, we run into a big problem. Consider $k=4$, so $\zeta_4 = i$. Then, consider the following: 
\begin{align*}
    &((i+i)+i)(i+i)+i\\
    =&((2i)+i)(2i)+i\\
    =&(3i)(2i) + i\\
    =&-6+i
\end{align*}
To salvage this system, we posit the following representation:
\begin{equation}\label{eqn:expressionpowers}
(((((x+x)+x)(x+x)+x^2)(x+x) + x^3) \ldots(x+x) + x^{k})(x+x) + x^1 +\ldots 
\end{equation}
This way, we do not run into the problem of adding different powers of $\zeta_k$. Note that because $x^{k + c} = x^{c}$ when $x = \zeta_k$, we can consider the binary expansion in groups of $k$ $x$'s. Then we can use the sum of an arithmetic sequence to see that to get each group of $k$ digits, we need $\frac{k(k+5)}{2}$ number of $x$'s. 

Recall that there are at most $ \left\lceil \frac{\lceil \log_2(n)\rceil + 2}{k} \right\rceil$ groups of $k$ digits. Thus it will take at most  $\left\lceil \frac{\lceil \log_2(n) \rceil + 1}{k} \right\rceil \times \frac{k(k+5)}{2}$ $x$'s, our proposition.
\end{proof}

We can attempt to use a similar method of generalizing base expansions to find an upper bound for the complexity of any element of $\mathbb{Z}[\zeta_k]$, not just natural numbers. We shall start by finding such a bound in the Gaussian Integers.

\begin{prop}
    For all $\alpha \in \mathbb{Z}[i]$, 
    \[
    ||z||_4 \leq 7\log_2|z| + 5
    \]
\end{prop}

\begin{proof}
This can be found by writing $\alpha$ in base $2i$, which has a complexity of 2. Meanwhile, the remainder with the maximum complexity is $1-i$, which has a complexity of 5.

Every time we divide $\alpha$ by 2i, we can get a quotient with a magnitude that is at most half of the magnitude of $\alpha$. Therefore, we can divide no more than $\log_2|\alpha|$ times. Meanwhile, the base $2i$ expansion has at most $\log_2|\alpha|+1$ digits, and each of these digits/remainders has a maximum complexity of 5. Thus, our base 2i expansion of $\alpha$ will use at most $2\log_2|\alpha| + 5(\log_2|\alpha|+1) = 7\log_2|\alpha| + 5$.
\end{proof}

\begin{remark}
    It should be noted that although there are associates of $1 - i$ that have less complexity (for example, $||i - 1|| = 3$), it is necessary to consider all associates. For example, if we are trying to divide $1-5i$ by $2i$, we could write it as $-2(2i)+(1-i)$, but we could also write it as $(-3-i)(2i)+(i-1)$. Even though the latter option has a less complex remainder, we would still have to utilize the former instead because we need our quotient to have a magnitude less than half of the magnitude of $1-5i$ for this technique to work.
\end{remark}

Unfortunately, if we attempt to use this method to create an upper bound for all elements of $\mathbb{Z}[\zeta_k]$, we run into a problem in that a division algorithm does not exist for most cyclotomic rings, since the remainder could be `larger' than the divisor. Furthermore, it is possible to add units in cyclotomic rings to get elements with magnitudes less than 1, so it is difficult to find a most complex remainder for a given element. Therefore, we shall define an alternative division algorithm in $\mathbb{Z}[\zeta_k]$. For the purposes of this proof, we will be using $b \zeta_k$, for some $b \in \mathbb{N}$, as the divisor and define this alternative division algorithm only for non-multiplex numbers (see Definition \ref{def:multiplex}). 

More rigorously, for $n \in \mathbb{Z}[\zeta_k], n = a_1 \zeta_k+a_2 \zeta_k^2 + \ldots + a_k \zeta_k^k$. Then, dividing by $b \zeta_k$ will give us a quotient of:
\[
    q= \sum _{i=1}^k \left \lfloor \frac{a_i}{b} \right \rfloor \zeta_{k}^{i-1}
\]
Then, the remainder would be the difference between $n$ and $q \cdot b \zeta_k$. It is okay for the magnitude of the remainder to be more than the divisor; what is important is that the coefficient for each power of $\zeta_k$ in the remainder is less than $b$. 

Using this, we can create an alternative Euclidean algorithm to reduce $n$. We can see that this algorithm terminates by applying the well-ordering principle on the constants. 

\begin{obs}
We acknowledge that terms in the expansion $n = a_1 \zeta_k+a_2 \zeta_k^2 + \ldots + a_k \zeta_k^k$ may cancel each other out in certain places. However, this is still a valid representation, as the coefficients may be 0. 
\end{obs}

\begin{theorem}
For any element of $n \in \mathbb{Z}[\zeta_k]$ and for $a_i \in \mathbb{N}$, we can express $n$ as:
\[n = a_1 \zeta_k+a_2 \zeta_k^2 + \ldots + a_k \zeta_k^k \]
Then the upper bound for the complexity of $n$ is: 
\[||n||_k \leq (2k - \frac{k}{p} +1)(\log_2|a_m| + 1) - 2\]
where $p$ is the smallest prime divisor of $k$ and $a_m = \max(a_1, a_2, \ldots, a_k)$. 
\end{theorem}
\begin{proof}
Consider applying the alternative Euclidean algorithm to $n$ with divisor $2 \zeta_k$. For some $b_i, r_i \in \mathbb{Z}[\zeta_k]$ satisfying the conditions of our alternative division algorithm, we have: 
\begin{align*}
    n &= 2\zeta_k \cdot b_1  + r_1\\
    b_1 &= 2\zeta_k \cdot b_2 + r_2\\
    & \vdots \\
    b_{s-1} &= 2 \zeta_k \cdot b_s + r_s\\
    b_s &= 2 \zeta_k \cdot 0 + b_s
\end{align*}
Since the reduction ends when all the coefficients of $\zeta_k^i$ of $b_s$ are less than 2, we can conclude that we will need $\log_2 a_m +1$ such equations. 

We can see that we can write
\[n=2\zeta_k(\ldots 2\zeta_k(2\zeta_k(2\zeta_k \cdot b_s+r_s) + r_{s-1}) + r_{s-2}) \ldots ) + r_1\]
Each time we multiply by $2\zeta_k$, we also add by the remainder, $r_i$. The coefficients for the roots of unity must be 1 or 0. Since we are looking for the `worst case scenario', we want to maximize the number of $\zeta_k$'s involved. However, if we sum all the roots of unity, they will cancel and we will end up with 0. Thus there is a limit to the number of roots of unity that we can include.

Let $p$ be the smallest prime divisor of $k$. Consider taking the powers mod $p$. Then, there would be $\frac{k}{p}$ elements in each residue class modulo $p$. If we sum every element of each residue class, we will obtain 0. Therefore, we must leave out one of these elements. Then, the maximal complexity we can obtain for $b_s$ (established later with the representation with Horner's algorithm, and note that this can obtain all possible most complex remainders) would be if we removed the first element in each residue class, namely $\zeta_k, \zeta_k^{2}, \ldots, \zeta_k^{k/p}$. We can see that once we remove these $p-1$ roots of unity, no other cancellation will occur. 

Just representing this sum of roots of unity by multiplying $\zeta_k$ many times itself is inefficient. Instead, we can represent $r_i$ with Horner's algorithm again, namely 
\[
    r_i = \zeta_k^{\frac{k}{p}}(\zeta_k + \zeta_k(\zeta_k + \zeta_k(\ldots + \zeta_k(\zeta_k + \zeta_k\zeta_k)))), 
\]
which uses $(k-\frac{k}{p}-1)2 + \frac{k}{p} +1 = 2k - \frac{k}{p} -1$.
So, at each reduction (of which there will be $\log_2 a_m + 1 -1 $), we will need $(2k - \frac{k}{p} -1)+2$ $\zeta_k$'s, in addition to another $(2k - \frac{k}{p} -1)$ $\zeta_k$'s for the final coefficient. Thus, we are able to represent any element $n \in \mathbb{Z}[\zeta_k]$ with at most $(2k - \frac{k}{p} +1)(\log_2|a_m| + 1) - 2$ elements, simplifying these terms. 
\end{proof}

Trivially, we can also see the following corollary for prime values of $k$:

\begin{corollary}
    For any element of $n \in \mathbb{Z}[\zeta_k]$ such that $k$ is prime, and for $a_i \in \mathbb{N}$, the upper bound for the complexity of $n$ is: 
\[||n||_k \leq 2k(\log_2|a_m| + 1) - 2\]
where $a_m$ is the $\max(a_1, a_2, \ldots, a_k)$. 
\end{corollary}

\begin{remark}
In the case where $k=4$, using our bound from alternative division, which turns out to be $7\log_2|a_m| + 5$, is more effective than using the earlier bound based on the magnitude: $a_m$ is always at most equal to the magnitude of a Gaussian Integer, and in most cases, it is smaller. This is because it is impossible to add multiples of distinct roots of unity in a non-reducible manner in the Gaussian Integers such that their sum has a smaller magnitude than one of the components. The same is also true for $k = 2, k = 3$, and, trivially, $k = 1$, but it is not true for larger values of $k$. This can be proven using the law of cosines, noting that the smallest angle between elements of $\mathbb{Z}[\zeta_k]$ on the complex plane is $\frac{2\pi}{k}$. In these cases, it would have been more valuable to have a magnitude-based upper bound.
\end{remark}

Using the idea of multiplying and adding to construct a representation for $n$, we attempt to create a tighter bound for the naturals. We take inspiration from Zelinsky's work in normal integer complexity and employ a method called ``switching bases.''

\begin{figure} [H]
    \centering
    \captionsetup{justification=centering,margin=1cm}
    \includegraphics[width = .75\textwidth]{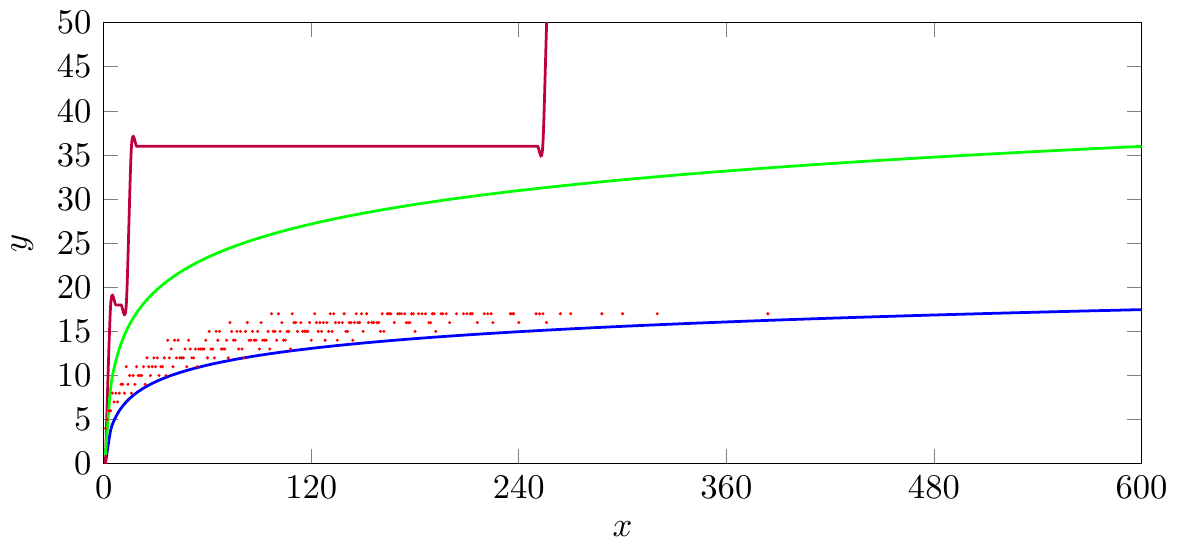}
    \caption{A plot of the complexity with $k=4$ ($y$) vs a natural number ($x$) up to complexity 18. The weak upper bound is shown in red, the stronger upper bound (inspired by Zelinsky) in green, and the lower bound in blue.}
    \label{fig:complexints}
\end{figure}

\begin{prop}
\[||n||_k \leq \frac{k(k-1)}{2} \cdot \left\lceil \frac{\lceil \log_2 n \rceil + 1}{k} \right\rceil + \dfrac{5}{\ln(2) + \frac{1}{2} \ln(3)} \ln(n)\]
\end{prop}
\begin{proof}

We attempt to generalize Zelinsky's `switching bases' method for higher $k$. This method still relies on expressing the number in a base representation, but instead of using only $2\zeta_k$ as a base, we switch bases when it is optimal. Let us use $2 \zeta_k, 3 \zeta_k, 5\zeta_k$. We claim that their complexities are 2, 3, and 5 respectively.

We can see that $||2 \zeta_k|| = 2$ and $||3 \zeta_k|| = 3$, which can be verified quickly since the only numbers of complexity 2 are $2\zeta_k$ and $\zeta_k^2$. For $5\zeta_k$, we can quickly check that the magnitude of anything with complexity 4 or less does not exceed 4.

In order to create an analogue to the natural numbers, we need to make sure that in the base expression, the remainder combines with the correct power of $\zeta_k$. Consider the base 2 representation of $n$. This expression will be the upper bound for the number of digits in the base $b$ expansion, thereby providing an upper bound for the number of times we need to account for different powers of $\zeta_k$ not simplifying into the same power of $\zeta_k$. This is analogous to how the powers of $x$ may not combine with the other addend in Equation \ref{eqn:expressionpowers}. There are $\lceil \log_2 n\rceil + 1$ digits total. Then, we partition the digits into groups of length $k$, where we group $a_1\zeta_k + a_2 \zeta_k^2 + \ldots + a_{k-1} \zeta_k^{k-1} + a_k$ together. We count the number of extra $\zeta_k$'s we will have to multiply to the remainder in order to add the coefficients, which is $\frac{k(k-1)}{2}$. This process is similar to the proof of Proposition \ref{prop:cycloweakupper}. Let us momentarily ignore the $2\zeta_k$ that we multiply at each step. Thus the upper bound for the number of extra $\zeta_k$'s needed is $\frac{k(k-1)}{2}\left\lceil\frac{\lceil \log_2 n \rceil + 1}{k}\right\rceil$.

Now to account for the complexity needed to do a bitshift in the multibase representation of $n$, we can simply consider the expression of $n$ as if it was in $k=1$, as we have accounted for extra powers of $\zeta_k$. Thus we add $\frac{5}{\ln(2) + \frac{1}{2}\ln(n)}\ln(n)$, as obtained by Zelinsky in his currently unpublished paper (used with permission) \cite{Zelinksky2022}.
\end{proof}

%
%
\begin{prop} \label{prop:comp1}
	$||1||_k = k$ for prime $k$.
\end{prop}

\begin{proof}
    Suppose we have an expression $f(x)$ such that $f(\zeta_k) = 1$. Then $\zeta_k$ is a root of the polynomial $f(x) - 1$. Since $f(x)$ is an expression, it has natural (a subset of the integers) coefficients and is divisible by $x$ (so the constant coefficient of $f(x)-1$ is $-1$). The minimal polynomial of an algebraic integer $\alpha$ divides all polynomials over $\mathbb{Z}$ with $\alpha$ as a root (\cite{Weisstein2002Aug}), so the minimal polynomial of $\zeta_k$ must divide $f(x) - 1$.
    
    Now consider when $k$ is prime. Then let $k=p$, where $p$ is prime. Then the minimal polynomial of $\zeta_p$ is the $p$-th cyclotomic polynomial, which is $x^{p-1} + x^{p-2} + \cdots + 1$, by \cite{Weisstein2002Apr}. Thus, for some polynomial $h \in \mathbb{Z}[x]$, we have:
    \[(x^{p-1} + x^{p-2} + \cdots + 1)h = f(x)-1.\]
    Then note that if $f$ has complexity at most $p$, then the largest degree of $f(x)$ is at most $p$, so $\deg h \le 1$. But the degree of $h$ can't be 0, because then $f(x)$ would equal $-x^{p-1} - x^{p-2} \cdots$, which isn't an expression \footnotemark. But, since the constant term must be $-1$ after the multiplication (and $f(x)$ has natural coefficients), $h$ has to be $ax-1$ for $a \in \mathbb{N}$. Thus $f(x) - 1 = x^p - 1 + (a-1)(x^{p-1} + \ldots) \implies f(x)=x^p + (a-1)(x^{p-1} + \ldots) $. But then, in order to have a degree of $p$, there has to be at least $p$ multiplications. Thus the complexity of $f$ is at least $p$, but the obvious expression has complexity $p$, so $||1||_k = k$.
\end{proof}

We have also found a second proof for this proposition, inspired by \cite{Lynn2022Sep}.
\begin{lemma}\label{lemma:coefEquality}
    Let $k$ be prime. Let $g$ be a generator in $\mathbb{Z}/k\mathbb{Z}$. Then if 
    
    \[S = a_0 + a_1\zeta_k^{g} + a_2\zeta_k^{g^2} + \ldots + a_{k-1} \zeta_k^{g^{k-1}} = a_0 + a_1\zeta_k^{g^2} + a_2\zeta_k^{g^3} + \ldots + a_{k-1} \zeta_k^{g^k},\]
    
    then $a_1 = a_2 = a_3 = \ldots = a_{k-1}$, for $a_i \in \mathbb{Z}.$
\end{lemma}

\begin{proof}
    First, move all terms of the equation to the same side. Then, noting that $\zeta_k^{g^k} = \zeta_k^{g}$ we have that 
    \[
    (a_1 - a_{k-1})\zeta_k^g + (a_2 - a1) \zeta_k^{g^2} + \ldots + (a_{k-1} - a_{k-2})\zeta^{g^{k-1}} = 0.
    \]
    Since $\{g, g^2, \cdots, g^{k-1}\} = \{1, 2, \cdots, k-1\}$, we can then reorder this equation to 
    
    \[b_1 + b_2 \zeta_k + b_3 \zeta_k^2 + \ldots + b_{k-2}\zeta_k^{k-2},\] 
    
    where each $b_i$ is some coefficient of a power of $\zeta_k$ from the equation prior (and note that each $b_i$ is a difference of 2 ``consecutive'' coefficients $a_j$ and $a_{j+1}$ where there is looping around) and from dividing both sides by $\zeta_k$. Thus $\zeta_k$ is a root of $g(x) = b_1 + b_2 x + b_3x^2 + \ldots + b_{k-2}x^{k-2}$.

    Then by \cite{Weisstein2002Aug}, $\phi_k$ divides $g(x)$. If $g \ne 0$, a contradiction arises as $\deg g < \deg \phi_k$. Thus $g=0$. Since the coefficients of $g$ are pairwise differences of ``consecutive'' coefficients of powers of $\zeta_k$ in $S$ and each coefficient of $g$ is 0 (as $g = 0$), all the coefficients $a_i$ for $k \ge i \ge 1$ are equal to each other.
\end{proof}

Now the alternate proof that $||1||_k = k$ for prime $k$ begins.

\begin{proof}
     Let $k$ be prime. Let $g$ be a generator in $\mathbb{Z}/k\mathbb{Z}$. Suppose $||1||_k < k$. Then there will be no $\zeta_k^k$'s in the expanded representation for $1$. Let the representation with all multiplication expanded be $a_{1} \zeta_k^g + a_{2} \zeta_k^{2} + \ldots + a_{k-1} \zeta_k^{g^{k-1}}$, noting that $\{g, g^2, \cdots, g^{k-1}\} = \{1, 2, \cdots, k-1\}.$
    
    Now consider the mapping $\zeta \rightarrow \zeta^g$ (which is a Galois automorphism, thus taking $1 \rightarrow 1$).
    
    Then 
    \[
    1 = a_1\zeta_k^{g} + a_{2}\zeta_k^{g^2} + \ldots + a_{k-1} \zeta_k^{g^{k-1}} = a_{1}\zeta_k^{g^2} + a_{2}\zeta_k^{g^3} + \ldots + a_{k-1} \zeta_k^{g^k}.
    \]
    
    By Lemma \ref{lemma:coefEquality}, $a_1 = a_2 = \cdots = a_{k-1}$. But since this is an expanded representation, all the coefficients are non-negative. If they are 0, then we get a contradiction $(1=0)$. If they were positive, then since $\zeta_k + \zeta_k^2 + \ldots + \zeta_k^{k-1} = -1$, we have have that $1 = -a_1 < 0$, a contradiction\footnotemark[\value{footnote}].
\end{proof}

\begin{remark}
    Using Galois automorphisms, we can trivially see that for an optimal representation of an integer, that representation is also optimal if we used other primitive roots of unity instead as the base. Furthermore, we can describe exactly what happens to the number represented if we switch the base to a different primitive root of unity.
\end{remark}

\footnotetext{
    Note the parallels between this proof and the other proof: they both used the fact that a representation of 1 (and in this case all rationals) using $\zeta_p$'s must have equal coefficients of the primitive root's powers.
    \label{footnote:galoisanalogues}
}

\begin{prop}\label{prop:2p, 1Comp}
    $||1||_k = k$ for $k = 2p$.
\end{prop}

\begin{proof}
    Note that $j = 1$ in this case (building on the idea from the proof of \ref{prop:2primepower1expression}. Now suppose that $||1||_k < k$. Then by Proposition \ref{prop:2primepower1expression}, $g_{p-1} \ge 1$. Thus, since $g_{p} = 0$ by Proposition \ref{prop:2primepower1expression}, the degree of $xf$ (the expression for $1$ is at least $2p-1$ if we expand the multiplications). Thus there are at least $2p-1$ multiplications to get such a degree. Therefore $||1||_k \ge 2p-1=k-1$.
    
    Since $||1||_k \le k$, if $||1||_k < k$, $||1||_k = k -1$. But since the degree of the expression for 1 is at least $k-1$, the only element st the degree is at least $k-1$ and complexity of $k-1$ is $x^{k-1}$. This leads to a contradiction, as clearly $\zeta_k^{k-1} \ne 1$.
    
\end{proof}

\begin{theorem}\label{theorem:primepowers1comp}
    	$||1||_k = k$ for prime power $k$.
\end{theorem}

\begin{proof}
    Let $k = p^{m+1}$ for $p$ prime and $m \in \mathbb{N}_0 = \mathbb{N} \cup \{0\}$. Furthermore, let $\mathbb{N}_0[x]$ be the set of polynomials with coefficients in $\mathbb{N}_0$. As before, the optimal expression for $1$ with $\zeta_{k}$ as a base implies that $\zeta_{k}$ is a root of $x f - 1$, for $f \in \mathbb{Z}[x]$. Then, since $x f-1 \in \mathbb{Z}[x], \phi_{k} | x f-1$ (\cite{Weisstein2002Aug}), where $\phi_{k}$ is the $k$-th cyclotomic polynomial (which just as a remainder, $k=p^{m+1}$).
    
    Thus there exists $ g \in \mathbb{Q}[x]$ such that $\phi_{k} g = x f - 1$. It is well known that
    \[
    \phi_k = \phi_{p^{m+1}} = 1 + x^{p^m} + x^{2p^m} + \ldots + x^{(p-1)p^m};
    \]
    therefore, the constant term is $1$, so the constant term of $g$ is $-1$. Thusly we can write $g = x g' - 1$ for $ g' \in \mathbb{Q}[x]$. Then let $g' = g_0 + g_1 x + g_2 x^2 + \ldots, \, g_i \in\mathbb{Q}$. Note that $\deg g' \le p^{m}-1$, since otherwise, $\deg xf \ge p^{m+1}$, and then there would be at least $p^{m+1}$ multiplications, which has a complexity greater than or equal to $p^{m+1}$ (and thus no counter-example to the claim can have this condition). Thus 
    \begin{align*}
        \phi_{k} (xg' - 1) =& \, x\phi_{k}g' - \phi_{k} \\
        &= x(1 + x^{p^m} + x^{2p^m} + \ldots+ x^{(p-1)p^m})(g_0 + g_1 x + g_2 x^2 + \ldots) \\
        &+ -(1 + x^{p^m} + x^{2p^m} + \ldots + x^{(p-1)p^m}).
    \end{align*}
    
    Then, 
    \[x f = x(1 + x^{p^m} + x^{2p^m} + \ldots + x^{(p-1)p^{m}})(g_0 + g_1 x + g_2 x^2 + \ldots) - x^{p^m} - x^{2p^m} - \ldots.\] 
    Clearly, $x f \in \mathbb{N}_0 [x]$, since its an expression. Thus all the coefficients of the RHS should be in $\N_0$ once simplified. Thus the coefficient of $x^{p^m}$ in the first/left term has to be greater than or equal to 1 (to account for the first subtraction). We find that this coefficient is $g_{p^m - 1}$. 
    
    
    Thus $g' = x^{p^m-1} + h, h\in\mathbb{Q}[x], \deg h \le p^m-1$ and the coefficient of $x^{p^m-1}$ in $h$ is non-negative. Then
    \begin{align*}
        \phi_{k}(x g' -1) &= \phi_{k}(x^{p^m} - 1 + xh)\\
        &= \phi_{p^{m+1}}\left(\prod_{a | p^m} (\phi_{a}) + xh\right)
    \end{align*}
   (This last step uses a well-known property of cyclotomic polynomials, \cite{Weisstein2002Apr}). 
   \begin{align*}
       \phi_{k}\left(\prod_{a|p^m} (\phi_{a}) + xh\right) &= x^{p^{m+1}} - 1 + xh\phi_{p^{m+1}}\\
        &= x f - 1 \\
        &\iff x f = x^{p^{m+1}} + x h\phi_{p^{m+1}}
   \end{align*}
   Thus the expression for $1$ using $\zeta_{p^{m+1}}$'s has to have a $\zeta_{p^{m+1}}^{p^{m+1}}$ term in it (as the coefficient of $x^{p{m+1}}$ is $g_{p^m-1}$, and this is positive) after all multiplication is distributed. Then suppose there is an expression of complexity less than $p^{m+1}$ for 1. Then the highest power of $\zeta_{p^{m+1}}$ explicitly in the expression is obviously less than $p^{m+1}$. But the condition shown before contradicts this, since this expression won't explicitly have a $\zeta_{p^{m+1}}^{p^{m+1}}$ term in it.
\end{proof}

\begin{remark}
    Considering polynomials not just over integers but over the cyclotomic rings might be fruitful in expanding the technique above.
\end{remark}

\begin{remark}
    It may be possible to generalize this technique to the other powers of $\zeta_k$ by considering the expression for $\zeta_k^d$ for $1 \le d < k$ to be $xf$, and $\zeta_k$ to be a root of $xf - x^d$ instead of subtracting $n$.
\end{remark}

\begin{prop}\label{prop:2primepower1expression}
    If $||1||_k < k$, then 
    \begin{align}
         g_{pj + j -1} &= 0, \label{prop2primepowerCondition1}\\
       g_{j-1}  &\ge 0, \label{prop2primepowerCondition2}\\
         g_{wj - 1} &\in \mathbb{N}_0 && \forall w \in\mathbb{N}, w \le p-1, \text{ and} \label{prop2primepowerCondition3}\\
         g_{pj-1} &\in \mathbb{N}, \label{prop2primepowerCondition4}
    \end{align}
    where $xg' - 1 = \frac{xf-1}{\phi_k}$, $g_a$ is the coefficient of $x^a$ in $g'$, $k = 2^a p^b$ for natural $a,b$, and $j = 2^{a-1} p^{b-1}$ and $g'$ satisfies this system of inequalities:
    \begin{equation}\label{prop2primepowercoefcondition}
    \begin{split}
        g_{j-1} + 1 &\ge 0 \\
        g_{2j - 1} - g_{j - 1} - 1 &\ge 0\\
        g_{3j - 1} - g_{2j - 1} + g_{j-1} + 1&\ge 0\\
        \vdots \\
        g_{(p-1)j} - g_{(p-2)j} + \cdots -g_{j-1} - 1 &\ge 0\\
        g_{pj - 1} - g_{(p-1)j} + \cdots + g_{j-1} &\ge 0\\
        - g_{pj - 1} + \cdots + g_{2j - 1} &\ge 0
    \end{split}
    \end{equation}
\end{prop}

\begin{proof}
    The following work is from a (failed) attempt to generalize the previous proof to $k$ a power of 2 times an odd prime power and a salvage to get the proposition above. The key step in the last proof was to find a condition necessary for $g \in \mathbb{Q}[x]$, one that would guarantee that $\phi_k g = x^k-1$. So if 
    \begin{align*}
        (x g' - 1)\phi_{k} &= x f - 1\\
        &= x^{k} - 1,
    \end{align*}
    then 
    \begin{align*}
        x g' - 1 &= \frac{x^{k} - 1}{\phi_{k}} \\
        &= \prod_{\substack{c \mid k \\ c \ne k}} \phi_c \\
        &=\prod_{d' | p^b} (\phi_{d'})\prod_{d | p^b} (\phi_{2d})\prod_{d | p^b} (\phi_{4d})\ldots \prod_{d | p^{b-1}} (\phi_{2^{a}d})\\
    \end{align*}
    
    First we consider when $d = 1$ in each of the products above. We can turn $\phi_2(x)$ to $-\phi_1(-x)$. Thus the product turns into
    \[
    \prod_{d' | p^b} (\phi_{d'})\left(-\phi_1(-x)\prod_{\substack{d | p^b\\ d \ne 1}} (\phi_{2d})\right) \left(-\phi_1(-x)\prod_{\substack{d | p^b\\ d \ne 1}} (\phi_{4d})\right)\ldots \left(-\phi_1(-x)\prod_{\substack{d | p^{b-1}\\ d \ne 1}} (\phi_{2^{a-1} d})\right).
    \]
    Since the remaining $d > 2$, we can apply the identity that $\phi_{2\ell}(x) = \phi_{\ell}(x^2), \ell \in\mathbb{N}$ \cite{Weisstein2002Apr} repeatedly to reduce all powers of 2 times $d$ to $2d$, resulting in
    \[
    \prod_{d' | p^b} (\phi_{d'})\left(-\phi_1(-x)\prod_{\substack{d | p^b\\ d \ne 1}} (\phi_{2d})\right)\left(-\phi_1(-x)\prod_{\substack{d | p^b\\ d \ne 1}} (\phi_{2d}(x^2))\right) \ldots \left(-\phi_1(-x)\prod_{\substack{d | p^{b-1}\\ d \ne 1}} (\phi_{2d}(x^{2^{a-1}}))\right).
    \]
    Finally, we can convert $\phi_{2d}$ to $\phi_d(-x)$ to get
    \[
    \left( \prod_{d' | p^b} (\phi_{d'}) \right) \left( (-1)\prod_{d | p^b} (\phi_{d}(-x)) \right) \left( (-1)\prod_{d | p^b} (\phi_{d}(-x^2)) \right) \ldots \left( (-1)\prod_{d | p^{b-1}} (\phi_{d}(-x^{2^{a-1}})) \right).
    \]
    
    Then, noting that $\prod_{d|p^b} \phi_d(-x^{2^{i}}) = (-x^{2^{i}})^{p^b} - 1$ by the product identity, we can substitute and obtain:
    \[
    xg' - 1 = (x^{p^b}-1)(x^{p^b}+1)((x^2)^{p^b}+1) \ldots ((x^{2^{a-1}})^{p^{b-1}} + 1).
    \] 
    
    Then using the difference of squares factorization (and letting $j = 2^{a-1}p^{b-1}$), we have
    \begin{align*}
        (x^{2^{a-1} p^b} - 1)(x^{2^{a-1}p^{b-1}} + 1) &= (x^{pj} - 1)(x^{j} + 1)\\
        &= x^{pj + j} + x^{pj} - x^{j} - 1
    \end{align*}
    Hence we want 
    \[
    xg' - 1 = x^{pj + j} + x^{pj} - x^{j} - 1 \iff g' = x^{pj + j - 1} + x^{pj - 1} - x^{j - 1}, x\ne 0 \text{ (trivially not the case)}.
    \]
    Thus the condition that we want for $g'$ is the coefficients of $x^{pj + j - 1}, x^{pj-1}$ in $g'$ are greater than or equal to $1$ and the coefficient of $x^{j-1}$ is greater than or equal to $-1$. This is sufficient to show that the optimal expression for $1$ is $x^{k}-1$ since if, for $h \in \mathbb{Q}[x]$
    \[
    g' = x^{pj + j -1} + x^{pj - 1} - x^{j-1} + h,
    \] 
    then 
    \[
        xg'-1 = x^{pj + j} + x^{pj} - x^j - 1 + hx.
    \]
    Hence
    \begin{align*}
        \phi_k(x)(xg' -1) &= (x^{pj + j} + x^{pj} - x^j - 1 + hx)\phi_k\\
        &= x^k - 1 + hx\phi_k.
    \end{align*}
    Thusly $x f = x^k + hx\phi_k$. Therefore $\zeta_k^k$ is in the representation of $1$ after all the multiplication is expanded. Thus at least $k$ multiplications are required, forcing the complexity of 1 to be greater than or equal to $k$. Since $\zeta_k^k = 1$, the complexity is in fact $k$. Thus this condition is all we need to show.
    
    Now to show this condition is true on the coefficients. As before, we know that 
    \[
    x f \in \mathbb{N}_0[x] \implies \phi_k(x)(xg' - 1) + 1 \in \mathbb{N}_0[x].
    \] 
    Thus $\phi_k(x)xg' - \phi_k(x)$ have coefficients in $\mathbb{N}_0$. Let $g' = g_0 + g_1x + \ldots$ as before. Since $k = 2pj$,
    \begin{align*}
        \phi_k &= \phi_{p^b} (-x^{2^{a-1}}) \\\
        &= \phi_p (-x^{j}) \\
        &= 1 -x^{j} + x^{2 \cdot j} -x^{3 \cdot j} + \ldots
    \end{align*}
    
     Since we subtract $\phi_k$, the coefficient of $x^{j}$ in $\phi_k(x) x g'$ has to be greater than or equal to $-1$, one of the conditions we need. Thus $a_{j-1} \ge -1$.
    
    Next distribute the $g'$ across $x\phi_k$ to obtain 
    \[
    g' x \phi_k = xg' -x^{j + 1}g' + x^{2j + 1}g' -x^{3j + 1}g' + \ldots.
    \] 
    Let $n = 1 + \deg g'$. Since the optimal expression for $1$ is going to have a degree less than or equal to $k =2pj$ (using the obvious expression), the highest possible degree of $\phi_k(x)(xg' - 1)$ is $(p-1)j + n$, which is less than or equal to $ 2pj$. 
    
    Thus $n \le pj + j$.
%
    
    Now consider finding the coefficients of $x^j, x^{2j}, x^{3j}, \ldots, x^{(p+1)j}$ in $\phi_k(x)(xg'-1) - \phi_k$. Since these coefficients are in $\mathbb{N}_0$, we get the following system of inequalities:
    \begin{align*}
        g_{j-1} + 1 &\ge 0\\
        g_{2j - 1} - g_{j - 1} - 1 &\ge 0\\
        g_{3j - 1} - g_{2j - 1} + g_{j-1} + 1&\ge 0\\
        \vdots \\
        g_{(p-1)j} - g_{(p-2)j} + \cdots -g_{j-1} - 1 &\ge 0\\
        g_{pj - 1} - g_{(p-1)j} + \cdots + g_{j-1} &\ge 0\\
        g_{pj+j - 1} - g_{pj - 1} + \cdots + g_{2j - 1} &\ge 0
    \end{align*}
    This gives us conditions \ref{prop2primepowercoefcondition}.
    Adding the third to last and second to last inequalities, we can see that $g_{pj-1} \ge 1$. We can also note that adding two adjacent inequalities up to the third to last line shows that for all $w \in \mathbb{N}$, $w \le p-1$, $g_{wj-1} \ge 0$.
    
    Now, if the expression for $1$ has a degree less than $k$, then by the contrapositive, $g_{pj + j - 1} \not\ge 1$. Then note that the LHS of each inequality above is equal to the coefficient of $x^j, \ldots$. Thus they must be integers, as they correspond to a coefficient in an expression. This gives us Equation \ref{prop2primepowerCondition3}. Since $g_{pj-1} \ge 1$, this gives us Equation \ref{prop2primepowerCondition4}. Then we can use induction to conclude that $g_{wj-1} \in \mathbb{Z} \, \forall w \in\mathbb{N}, w \le p +1$. Then, $g_{pj+j-1}$ is non-negative since it equals a coefficient of an expression (namely the highest degree one), and it is not greater than or equal to 1, so $g_{pj+j-1} = 0$. This gives us Equation \ref{prop2primepowerCondition1}.
\end{proof}

Note that this proposition provides a sufficient condition for generating expressions for $1$: we can simply find $g_{j-1}, g_{2j-1}, g_{3j-1}, \ldots, g_{(p-1)j - 1}$ with the conditions in Proposition \ref{prop:2primepower1expression} and let the other coefficients be 0, and we get that $\phi_k(x) \cdot (xg' - 1)$ corresponds to an expression for 1. For instance, if $k=6$, we get this system of inequalities:
\begin{align*}
    g_0 + 1 \ge 0\\
    g_1 - g_0 - 1 \ge 0\\
    g_2 - g_1 + g_0 \ge 0\\
    g_3 - g_2 + g_1 \ge 0
\end{align*}

Then if $||1||_k < k$, we get this:
\begin{align*}
    g_0 + 1 \ge 0\\
    g_1 - g_0 - 1 \ge 0\\
    g_2 - g_1 + g_0 \ge 0\\
    - g_2 + g_1 \ge 0
\end{align*}

Then we can see that $g_0, g_1, g_2 = 0, 1, 1$, among other solutions, satisfies the system. This corresponds to $xf = x + x^5$. This is an expression for $1$ with $k=6$.

For another example, consider $k = 50$. We get this:

\begin{align*}
    \vdots\\
    g_{19} - g_{14} + g_9 - g_4 - 1 &\ge 0\\
    g_{24} - g_{19} + g_{14} - g_{9} + g_4 &\ge 0\\
    - g_{24} + g_{19} - g_{14} + g_9 &\ge 0
\end{align*}

Then note that if $g_{9}, g_{14}, g_{19}, g_{24} = 1$ and the rest are 0, we get that $xf = x^{5} + x^{15} + x^{35} + x^{45}$. This has a complexity less than or equal to $47 < 50$, since it can be written as $x^4(x + x^{10}(x+x^{20}(x+x^{11})))$.

\begin{remark}
    In the optimal expression for $n$, we can get a polynomial $g'$ as above, and with some algebra we can get that $g' = x^{q^{b-1}-1}$ for $k=q^b$.
    
    As before, if we have an expression $\zeta_k f(\zeta_k) = n$ for $f \in \mathbb{N}_0[x]$, then $\phi_k \mid x f - n$. Thus $\exists g \in\mathbb{Q}[x]$ such that $\phi_k(x) (x g - n) = x f - n$. Then we note that $\phi_{q^b} = \sum_{0 \le a \le p-1} x^{aq^{b-1}}$, and let $g = \sum_{i = 0} g_i x^i$. Then from distributing and equating to $x f$, we can see that $n + x\phi_k(x) g - n\phi_k \in \mathbb{N}_0[x]$. Substituting (and doing a little cancellation), we get
    \[
    x\left( \sum_{0 \le a \le p-1} x^{aq^{b-1}} \right) \sum_{i = 0} g_i x^i - n\sum_{1 \le a \le p-1} x^{aq^{b-1}}.
    \]
    Doing more distributing, we get
    \[
    \sum_{0 \le a \le p-1} \sum_{i = 0} g_i x^{1 + aq^{b-1} + i} - n\sum_{1 \le a \le p-1} x^{aq^{b-1}}.
    \]
    Then the coefficient of $x^{q^{b-1}}$ in the first addend has to be greater than or equal to $n$ for the polynomial to be in $\mathbb{N}_0[x]$. Thus $g_{q^{b-1}-1} \ge n$.
\end{remark}

\begin{corollary}
    The above remark fairly immediately implies that $||n||_{q^{b}} \ge q^b$ based on degree analysis.
\end{corollary}

\begin{remark}
    By applying a similar technique to $k=1$ of keeping track of the coefficients and forcing them to be in the naturals, we can easily see that (although note that $nx = n$ in this case, there is no subtraction, and a sum of $x + x^2 + \ldots + x^n = n$ in this case as well, forcing the degree of $xf$ to be $\le n$).
    \[
    n \ge g_{0} \ge g_1 \ge g_2 \ldots g_n \ge 0.
    \]
\end{remark}


\begin{prop}\label{conj:unitcomplexity}
    We originally had the conjecture that for all $a, k \in \mathbb{N}, a \le k$,
    \[
    ||\zeta_k^a||_k = a.
    \]
    But this turns out to be false.
\end{prop}
Though this seems to be correct intuitively, there seems to be increasingly more counterexamples as $k$ increases in the number of distinct prime factors (since cyclotomic polynomials underline the structure of the expression and having more factors makes them messier). A counterexample is $||1||_{12} \leq 11$ if we consider $(\zeta_{12}^9 + \zeta_{12}) \zeta_{12} = 1$.


\begin{prop}\label{prop:half1comp}
    $||\zeta_k^a||_k = a$ for $a \le \frac{k}{2} + 1$.
\end{prop}
    
\begin{proof}
	We will use induction. First we assume that \(\forall b \in \mathbb{N}, b \le a \) for a fixed natural \(a \) with our condition, the vector in the complex plane with the largest angle and complexity \(b \) is $e^{2\pi i b/ k}$. We can see that for \(a=1 \), this is true.

	Then, we can show this is true for \(a+1 \) through contradiction. Suppose that there was a vector with an angle greater than $\frac{2 \pi i (a+1)}{k}$ with complexity $a+1$, call it $\alpha$. Then, the final operation performed must be a multiplication, since adding two vectors results in a vector with an angle between the addend's angles and the angle between the two vectors is less than $\pi$ radians, as we can't get a vector with an angle less than $\frac{2\pi i}{k}$ (multiplication can only increase the angle; angle addition can also only produce a vector with an angle larger than the smallest angle of these addends).
	
	Thus let $\alpha = \beta \gamma$. Since multiplication adds angles, $\arg \beta + \arg \gamma > \frac{2\pi (a+1)}{k}$. Let $b = ||\beta||_k$ and $c = ||\gamma||_k$. Note that $b+c = a+1$ since $\alpha = \beta \gamma$ in its optimal expression. Thus the largest angles that $\beta$ and $\gamma$ can have are $\frac{2\pi b}{k}$ and $\frac{2\pi c}{k}$ respectively. This leads to a contradiction as $\arg \alpha = \arg \beta + \arg \gamma = \frac{2\pi (b+c)}{k}$ and $b+c = a + 1$.
\end{proof}

\begin{remark}
    We can also get $a = \lfloor \frac{k}{2} \rfloor + 2$ by considering that two vectors must once again add to have an angle of $\frac{2\pi i a}{k}$, but there are no vectors with complexity less than $a$ with an angle of $\pi$ between it and $\zeta_k$ going counter-clockwise from $\zeta_k$.
\end{remark}

\begin{prop}\label{prop:0comp,evenk}
    If $k$ is even, $||0||_k = \frac{k}{2} + 2$.
\end{prop}

\begin{proof}
    Clearly, the last operation in the optimal representation for $0$ isn't multiplication. Thus there must be at least two vectors that add up to 0 in the representation of 0. Clearly, these vectors aren't all in quadrant 1 or 2 (or on the $x$-axis, unless $k=2$). Thus there is a vector in quadrant 3 or 4. By Proposition \ref{prop:half1comp}, the complexity of such a vector is at least $\frac{k}{2}+1$. The least complex vector in the first or second quadrant has a complexity of 1. Thus $\zeta_k + -\zeta_k$ is the least complex representation for $0$ at a complexity of $k$, which has a complexity of $\frac{k}{2}+2$.
\end{proof}

\begin{prop} \label{prop:classificationOf1Comp}
For even $k$, let $p$ be the smallest odd prime factor of $k$. Then
\[
||1||_k \leq k-\frac{k}{2p} + p-2
\]
and if $k > 2p^2 - 4p$,
\[
||1||_k \leq k-\frac{k}{2p} + p-2 < k
\]
\end{prop}
\begin{proof} 
    Since $k$ is even, let $2m = \frac{k}{p}$ for $m \in\N$. The following is then true:
    \[ 
    \zeta_{2p} + \zeta_{2p}^3 + \dots + \zeta_{2p}^{2p-1} = 0.
    \]
    Then, to get 1, we can simply leave out $\zeta_{2p}^p = -1$. 
    \[
    \zeta_{2p} + \zeta_{2p}^3 + \dots + \zeta_{2p}^{2p-1} - \zeta_{2p}^{p} = 1.
    \]
    Note that this is an expression. If we were to translate this to base $k = (2p)m$, we would have 
    \[
    \zeta_k^{m} + \zeta_k^{3m} + \dots + \zeta_k^{k-m} - \zeta_k^{mp} = 1
    \]
    By Horner's algorithm, the complexity of this is at most $k - m + p -2$ (follows from Proposition \ref{prop:degreeHorner}):. Then if $||1||_{k} < k$,
    \begin{align*}
        k - \frac{k}{2p} + p - 2 &< k\\
        \frac{k}{2p} &> p-2\\
        k &> 2p^2 - 4p.
    \end{align*}
    Since Horner's algorithm is an upper bound, we also obtain
    \[
     ||1||_k \leq k - \frac{k}{2p} + p - 2.
    \]
\end{proof}

\begin{corollary}\label{corollary:reverse1Comp}
    If $k = 2^a p^b$ for $a \ge 1, b > 1$, then 
    $k-\frac{k}{2p} + p -2 = 2^a p^b - 2^{a-1} p^{b-1} + p - 2$. 
    Since $a \ge 1,b > 1, \, p-2 - 2^{a-1}p^{b-1} < 0 \implies k-\frac{k}{2p} + p - 2 < k \implies ||1||_k < k$.
\end{corollary}

\begin{remark}
    Interestingly, for even \textit{k}, when we represent $\zeta_k^{m} + \zeta_k^{3m} + \dots + \zeta_k^{k-m} - \zeta_k^{mp} = 1$ using Horner's method, the highest power of $\zeta_k$ used in the expression is $\zeta_k^{2m}$, which according to Proposition \ref{prop:half1comp} indeed has a complexity of $2m$ (as $2m = \frac{k}{p} \le \frac{k}{2}$).
\end{remark}

\begin{corollary} \label{corollary:1compPrimeBound}
    If $k = 2^a p^b$ for an odd prime $p$, $a,b \in\N$, and $||1||_k = k$, then $p > 2^{a-2}$.
\end{corollary}

\begin{proof}
    If $k = 2^a p^b$, then when $||1||_k = k$, $k = 2^a p^b \le 2p^2 - 4p$ by Proposition \ref{prop:classificationOf1Comp}. Therefore $2^{a-1} p^{b-1} \le p - 2$. Thus $p(2^{a-1} p^{b-2} - 1) \le -2$, and then $2^{a-1} p^{b-2} - 1 \le \frac{-2}{p} \le 0 \implies 2^{a-1} p^{b-2} \le 1$. Therefore $2^{a-2} \le p^{2-b}$. If $b = 2$, then for Corollary \ref{corollary:reverse1Comp} applies (since $a\in\N$), and we get a contradiction. Thus $b\ne 2$. Thus $b=1$. Thus we have that $2^{a-2} \le p$.
\end{proof}

\begin{conj}\label{conj:classificationof1comp}
    If $k=2^a p^b$, $||1||_k = k \iff a = 1, b = 1$ or $a=0$.
\end{conj}

Notice that the above conjecture has only the case in which $b=1$ and $a\ne 1$ as unestablished as to what the complexity of 1 is.

\begin{conj} \label{conj:onebestwecando}
    For even $k > 2p^2 - 4p$ where $p$ is the smallest odd prime divisor of $k$,
    \[
    ||1||_k = k - \frac{k}{2p} + p - 2.
    \]
\end{conj}

\begin{conj} \label{conj:whenisoneone}
    For even $k \leq 2p^2 - 4p$ where $p$ is the smallest odd prime divisor of $k$,
    \[
    ||1||_k = k.
    \]
\end{conj}
This conjecture essentially states that Lemma \ref{prop:classificationOf1Comp} is a double implication. 

To further explore complexity in cyclotomic rings, let us define the following set:
\begin{definition}
    A number is \textbf{minimal} when there is an optimal representation that does not contain the multiplication of two multiplex numbers. The set of all minimal numbers is denoted by $\mathbb{M}$.
\end{definition}
Where multiplex is defined as follows:
\begin{definition}\label{def:multiplex}
    For the purposes of this paper, a \textbf{multiplex} number is an element of the set $\mathbb{Z}[\zeta_k] \setminus \{a \zeta_k^n: a \in \mathbb{N}, n \in \mathbb{N}\} $. 
\end{definition}
The intuition for this is that there are at least 2 non-zero coefficients.

For example, these are some minimal numbers in $k=4$:
\begin{align*}
    2&=i^3\cdot 2i\\
    2+2i &= 2i + 2i\cdot i^3.
\end{align*}
Most numbers are minimal, but there are still numbers that are nonminimal. For example, the two least complex nonminimal numbers for $k = 4$ are
\begin{align*}
    3-4i &= (i-2)^2 = (i + i(i + i))^2\\
    3 - 5i &= (i-1)(i - 4) = (i + i\cdot i)(i + (i + i)(i + i)),
\end{align*}
both with complexity 8, which one can verify with a little bit of coding.

Investigating minimal numbers, we come up with the following propositions and conjectures.

\begin{prop}
$\mathbb{M}$ is closed under ``optimal addition'': \\
For $a,b \in \mathbb{M}$ and their optimal representations $A,B$ respectively, if $A+B$ is an optimal representation of $a+b$, then $a+b \in \mathbb{M}$
\end{prop}
\begin{proof}
This follows from the definition. 
\end{proof}
\begin{remark}
It is important to note that the converse might not be true. 
\end{remark}

\begin{conj} \label{conj:minimality}
All non-multiplex numbers are minimal.  
\end{conj}

\begin{remark}
    By testing the first 16 complexities for the Gaussian Integers, we have yet to find a counterexample to this, or to the more specific conjecture that all natural numbers are minimal. However, it should be noted that the first non-minimal Gaussian Integer has complexity 8, so we may have not tested enough complexities to find one that happens to be a natural number.
\end{remark}

\begin{conj} \label{conj:bestTheoreticalUpper}
The largest value of $\frac{||n||_k}{\ln n}$ for a fixed $k > 1$ is $\frac{||2||_k}{\ln 2}$.
\end{conj}
\begin{remark}
We have verified this for $k = 2$ up to complexity 26, $k=4$ for complexities up to $26$, and $k=3$ for complexities up to 10. This is the analogue of $\frac{26}{\ln 1439}\ln n$ as the best theoretical upper log bound for $k=1$.
\end{remark}

\begin{conj}\label{conj:powerliftingBound}
    \[
    ||n||_{p^k} \le ||n||_{p^{k+1}}
    \]
    for $p$ a prime and $k$ a natural.
\end{conj}

\begin{remark}
    Some proof approaches we have attempted is to assume Conjecture \ref{conj:minimality}, replace an optimal representation using $\zeta_{p^{k+1}}$ with $\zeta_{p^k}$. The primary issue with this approach is that while this forces addends to have the same power of $\zeta_{p^{k+1}}$, this doesn't guarantee that we can have an expression of the proper associate of $n$ using the $\zeta_{p^k}$. For instance, $7 = i (i + (i + i)^3)$ but $7 =  -1 + -1 \cdot (-1 + -1)^3$ (the first expression `passes through' an associate of 7, but the second obtains it directly).
    
    Another approach is to show that the extra degrees of freedom granted by allowing the vectors corresponding to powers of $\zeta_{p^{k+1}}$ can't allow for a ``shorter'' path.
\end{remark}

\begin{remark}
    It may also be interesting to look at cyclotomic rings in terms of rationality. If a representation represents a natural number, the sum of the imaginary components must be 0. For most $k$ and $p$, this imaginary component will be irrational, specifically outputs of the sine function on a rational multiple of $2\pi$. Thus the sum of the imaginary components of the vectors composing the optimal expression may have structure that can be analyzed. It also may be true that if $a$ is in the expression for a natural, the conjugate must also, since we need the sum of the imaginary components to equal 0 and they are irrational. This would then imply the minimality conjecture (\ref{conj:minimality}).
\end{remark}

\section{Polynomial Rings} \label{Polynomial Rings}
The next step for generalization is to turn into the polynomial ring $\mathbb{Z}[x]$. Since our expressions are written by summing and multiplying $x$'s, it would assist our exploration to investigate this ring. 

\begin{obs}\label{obs:basicPolyProps}
We can see that for $k \in \mathbb{N}$,
\begin{align*}
    ||x^k|| &= k\\
    ||kx|| &= k
\end{align*}
\end{obs}

An important application of polynomial rings is to quickly convert expression to a representation. Turning an expression to a representation can be done by substituting a base $k$ for $x$.

For some example data, see Table \ref{tab:s_k}.

\begin{table}
    \centering
    \begin{tabular}{c|c} 
    $k$ & $n$ \\
    1 & $x$ \\
    2 & $x^2, 2x$ \\
    3 & $x^3, x^2 + x, 2x^2, 3x$ \\
    4 & $x^4, x^3 + x, x^3 + x^2, x^2 + 2x, 2x^3, 2x^2 + x, 3x^2, 4x, 4x^2$ \\
    5 & $\cdots 5x$\\
    \vdots 
    \end{tabular}
    \caption{A table of expressions $n$ with complexities $k$.}
    \label{tab:s_k}
\end{table}

We investigate the bounds, not on a specific polynomial, but for the number of polynomials with complexity $k$. For each polynomial of complexity $a$ and complexity $b$, it is not always possible to create a new polynomial of complexity $a + b$ just by adding or multiplying the two simpler polynomials. For example, $2x + 2x = 4x$, but this is a repeat of the expression $x + 3x$, a sum that also implies a complexity of 4. In general, it is difficult to calculate the size of the complexity classes because of these repeats.

\begin{definition}
    Let the $k$th \textbf{complexity class} (denoted by $K_k$) be the set of all polynomials with complexity $k$
\end{definition}

\begin{table}[H]
\centering
\captionsetup{justification=centering,margin=1cm}
    \begin{tabular}{|c|c||c|c|}
        \hline 
        $n$ & $|K_n|$ & $n$ & $|K_n|$\\
        \hline
        1 & 1    &    8 & 252\\
        2 & 2    &    9 & 652\\
        3 & 4    &    10 & 1554\\
        4 & 9    &    11 & 3967\\
        5 & 19   &    12 & 10280\\
        6 & 45   &    13 & 27008\\  
        7 & 104  &    14 & 71738\\
        \hline
    \end{tabular}
    \caption{
\label{tab:polycomplexities} Table of first few complexity classes $K_n$, where $n$ is the complexity value and $|K_n|$ is the number of elements in the complexity class.}
\end{table}

\begin{prop}\label{prop:upperboundCompclass}
For all $n \in \mathbb{Z}$, 
\[
|K_n| \leq  n^{n-1}
\]
\end{prop}

\begin{proof}
    In order to find an upper bound on these complexity classes, we may simplify our counting and over estimate by treating it as if none of the combinations of elements are repeats. Using this, it can be seen that 
    \[
    |K_n| \leq 2\Bigl(|K_1|\cdot |K_{n-1}| + |K_2|\cdot |K_{n-2}| + \ldots + |K_{n/2}| \cdot |K_{n/2}| \Bigr )
    \]
    or $|K_{(n-1)/2}| \cdot |K_{(n+1)/2}|$ depending on the parity of $n$. Observe that the RHS can be rewritten as $\displaystyle \sum_{i=1}^{n-1} \Bigl( |K_i|\cdot |K_{n-i}| \Bigr)$.
    
    The proof will be conducted by induction. Firstly, observe that $|K_1| = 1 \leq 1^{1-1}$ as a base case.

    From here, we shall assume that for all $1 \leq i < n, |K_i| \leq i^{i-1}$. Then it follows that $|K_i| \leq n^{i-1}$. Then, we can substitute to find that 
    \[
    |K_n| \leq  \sum_{i=1}^{n-1} n^{i-1}\cdot   n^{n-i-1} = \sum_{i=1}^{n-1}  n^{n-2} =  (n-1)\cdot n^{n-2} \leq  n^{n-1}.
    \]
\end{proof}

Here we have proved the trivial bound. Admittedly, this is a very weak bound, considering that for the most part, in our experience, the size of complexity classes seem to grow exponentially with a fixed base. For example, letting $||n|| = c$ and using our upper and lower bounds of $\frac{3}{\ln 3} \ln n \leq c \leq \frac{3}{\ln 2} \ln n$ for the natural numbers, we can find that any $n$ with complexity $c$ must satisfy the inequality $2^{c/3} \leq n \leq 3^{c/3}$, giving us a restriction of $3^{c/3} - 2^{c/3} + 1$ on the size of the $c$-th complexity class.

We can attempt to take inspiration from this idea by utilizing the bounds on polynomials with complexity $c$. The following two lemmas, in particular, may be of use:

\begin{lemma} \label{lem:3.2}
    A polynomial of complexity $c$ must have a degree of at most $c$.
\end{lemma}

\begin{proof}
    This is trivial, since any polynomial of degree $d$ must have at least $d$ $x$'s multiplied together.
\end{proof}

\begin{lemma}\label{lem:3.3}
    A polynomial of complexity $c$ whose sum of coefficients is $s$ satisfies $s \leq 3^{c/3}$.
\end{lemma}

\begin{proof}
    This can be seen from the fact that the sum of coefficients of a polynomial $f(x)$ is equal to $f(1)$. We can then treat this polynomial as a complexity expression. We know from regular integer complexity that the maximum number with complexity $c$ is $3^{c/3}$.
\end{proof}

While we are proving such restrictions, we can also verify the following upper bound on complexity, although we won't use it to find a restriction on the size of complexity classes:

\begin{lemma}\label{lem:deg}
    If $f(x)$ has degree $d, \, ||f(x)|| \leq f(1) + d - 1$
\end{lemma}

\begin{proof} This result can be obtained by writing $f(x)$ in ``base $x$'' using Horner's method. For example: \[
3x^4 + 2x^3 + 4x^2 + 2x = x + x + x\left( x + x + x + x + x\left( 
x + x + x\left(  x + x + x \right) \right) \right).
\]
Recall that $f(1)$ is equal to the sum of the coefficients.
\end{proof}

\begin{remark}
Note that using Horner's algorithm shows that complexity in terms of $x$ is defined for all polynomial functions of $x$ with positive integral coefficients that doesn't have a constant term.
\end{remark}

Now, we can use the lemmas \ref{lem:3.2} and \ref{lem:3.3} to find a restriction on how many polynomials could have a complexity of $n$. That process leads us to the following bound:

\begin{prop}
    Let $K_n$ be the class of polynomials with complexity $n$. Then 
    \[\sum_{j=1}^{n}|K_j| \leq \sum_{i=1}^{\lceil 3^{n/3}\rceil} \binom{i+ n}{n}.\]
\end{prop}

\begin{proof}
    Any polynomial with a complexity up to $n$ can have at most degree $n$ and sum of coefficients $3^\frac{n}{3}$. Therefore, in order to find the number of polynomials with positive integral coefficients and no constant term, we need to find the number of ways to ``distribute'' $i$ ``ones'' among $n$ coefficients for $i \leq 3^{n/3}$. As we range $i$ from 1 to $\leq 3^{n/3}$, we end up counting all the different polynomials that we could obtain. Using the stars and bars method, we get a maximum of $\binom{i + n}{n}$ possible polynomials and we need to sum all the possible $i$'s. 
\end{proof}

Admittedly, there are many setbacks to this method that makes it a weak bound. It severely overcounts possible polynomials. For example, the only polynomial with degree $n$ and complexity $n$ is $x^n$ itself, but our method considers all $n$th degree polynomials with coefficients summing to at most $3^\frac{n}{3}$.

Furthermore, it turns out that for sufficiently large $n$, this bound is weaker than $\sum_{i=1}^ni^{i-1}$, our previous bound. Verifying this will be left as an exercise to the reader.

Meanwhile, we can much more effectively attempt to create a lower bound for $|K_n|$ by considering a modified version of complexity that is easier to count:

\begin{definition}
    Let the \textbf{modified complexity} of a polynomial $f(x) \in S$ be the number of $x$'s it takes to write $f(x)$. $S$ is defined as the following 
    \begin{itemize}
        \item $x \in S$. 
        \item If $f(x) \in S$, then $f(x) + x \in S$
        \item If $f(x) \in S$, then $f(x) \cdot x \in S$
    \end{itemize}
    Denote the modified complexity of $f(x)$ as $|||f(x)|||$. \\
    Note $S$ is a set of expressions, not a set of polynomials. They look similar, but they are not the same. 
\end{definition}

For example, 
\[
x(x \cdot x + x \cdot x + (x + x + x)) 
\] 
is not a valid modified complexity expression because it has a non-nested multiplication (namely two $x\cdot x$'s), but 
\[
x(x + x + x + x(x + x))
\]
is. Though they evaluate to be the same polynomial, the latter is in a valid form while the former is not. 

\begin{remark}
    Following from the definition, we can see that any expression of modified complexity is equivalent to Horner's method, which only adds or multiplies $x$'s to expressions, but does not add or multiply together two subexpressions that both have complexity greater than 1.
\end{remark}

\begin{lemma}
    Every polynomial with positive integral coefficients and no constant term has a unique modified complexity expression.
\end{lemma}

\begin{proof}
    We already know that Horner's method is able to express all polynomials with positive integral coefficients and no constant term in terms of $x$. Therefore, we can conclude that every such polynomial has a modified complexity expression, which is equivalent to the Horner's method expression. Furthermore, since Horner's method amounts to writing a polynomial in base $x$, it follows that there is only one modified complexity expression for each polynomial since there is a unique base $x$ expression for each polynomial.
\end{proof}

\begin{lemma}
    Let $K'_n$ be the $n$-th modified complexity class. Then 
    \[
    |K'_n| = 2^{n-1}
    \]
\end{lemma}

\begin{proof}
    Trivially, the only member of the first modified complexity class is $x$, so $K'_1 = 1$.
    
    Next, assume that $K'_n = 2^{n-1}$. From each polynomial $f(x)$ in that complexity class, we get two polynomials in $K'_{n+1}$: $f(x) + x$ and $f(x) \cdot x$. From the above lemma, we know that each one of these polynomials is unique, so there is no possibility of any of them having appeared in any earlier complexity class.
    
Meanwhile, FTSOC suppose there was a polynomial of modified complexity $n+1$ not in this list. By WOP, let $n+1$ be the smallest such. But then, that polynomial can be reduced to a polynomial of complexity $n$ by just undoing the last operation in the expression. Therefore, we can be sure that the $2^n$ polynomials in $K'_{n+1}$ are the only polynomials in that modified complexity class.

Thus, by induction, $K'_n = 2^{n-1}$.
\end{proof}

With this in mind, we have enough to find a bound on the cumulative sum of the sizes of the first $n$ complexity classes, although this doesn't imply a bound on the exact number of elements per complexity class.

\begin{prop} \label{prop:modifiedlowerboundcompclass}
For $n \geq 1$,
\[
2^n - 1 \leq \sum_{i = 1}^{n}|K_n|
\]
\end{prop}
\begin{proof}
First note that $|K'_n| \leq |K_n|$, since given an element of $K'_n$, it has complexity less than or equal to $n$, as the modified complexity express is still an expression, it just may not be optimal. Thus every element of $K'_n \subset K_n$. Thus the size of $K_n$ is greater than or equal to the size of $K'_n$. Since the size of $K'_n$ equals $2^{n-1}$, our hypothesis is easily obtained with a geometric sum.

\end{proof}

Even though we cannot find a bound for each of the complexity classes by this method, we can still conjecture one:

\begin{conj} \label{conj:2lowerbound}
For $n \geq 1$,
\[
2^{n-1} \leq |K_n|
\]
\end{conj}

In addition to this, plotting the first fourteen values of $|K_n|$ supports the conclusion that the class size grows exponentially, approximately. We can also gain insight by looking at the ratios $\dfrac{|K_n+1|}{|K_n|}$:

\begin{table}[H]
\centering
\captionsetup{justification=centering,margin=1cm}
    \begin{tabular}{|c|c||c|c|}
        \hline 
        $n$ & $\dfrac{|K_{n+1}|}{|K_n|}$ & $n$ & $\dfrac{|K_n+1|}{|K_n|}$\\
        \hline
         1 & 2      &    7 & 2.423\\
         2 & 2      &    8 & 2.460\\
         3 & 2.25   &    9 & 2.506\\
         4 & 2.111  &    10 & 2.553\\
         5 & 2.368  &    11 & 2.591\\
         6 & 2.311  &    12 & 2.627\\  
        \hline
    \end{tabular}
\end{table}

If this sequence of terms ends up converging to a constant value, the class size function will approach an exponential function, and this exponential function could be used as an upper bound. It is not unreasonable to guess that perhaps this is the case, and it may even be possible that the constant in question is $e$.

\begin{conj} \label{conj:eUpperbound}
    For $n \geq 1$,
\[
e^{n-1} \geq |K_n|
\]
\end{conj}

\begin{figure}[H]
    \centering
    \captionsetup{justification=centering,margin=0.5cm}
    \includegraphics[width=.6\textwidth]{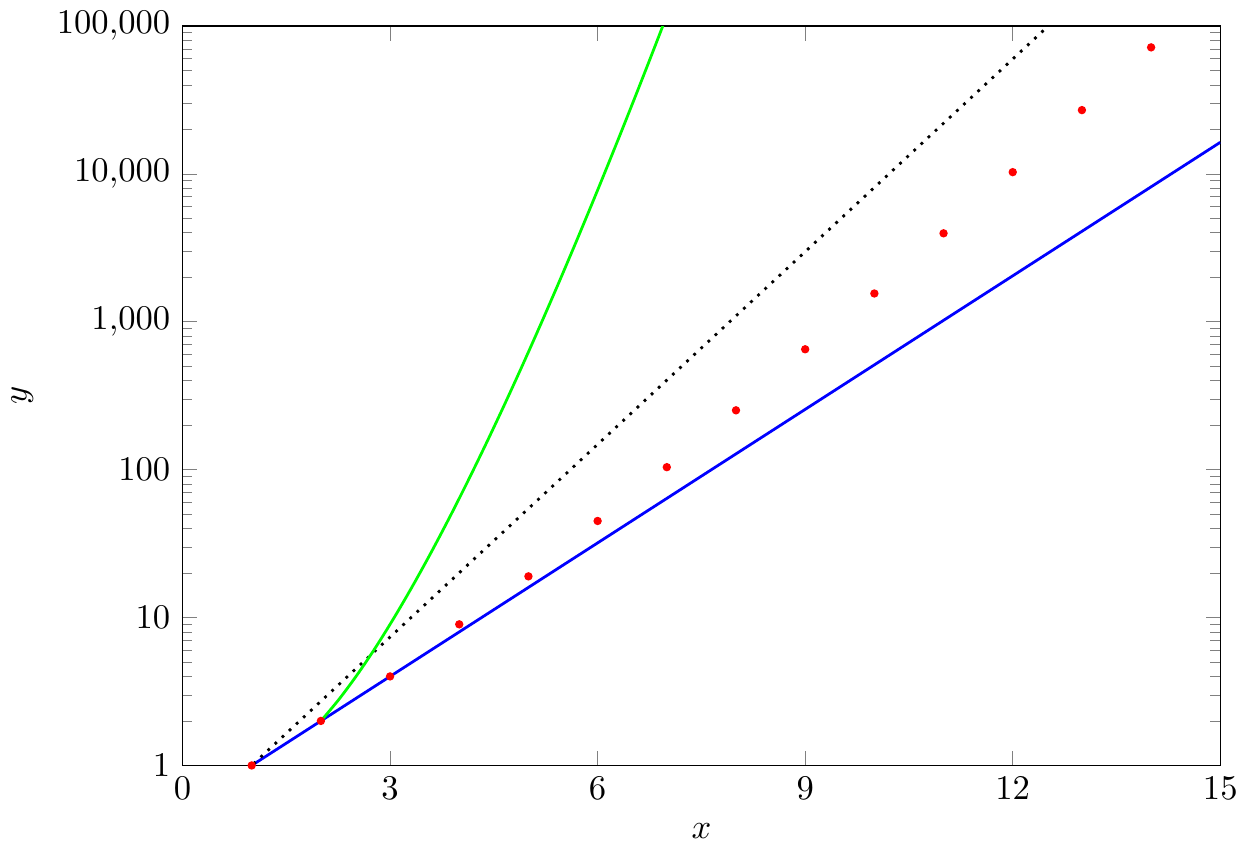}
    \caption{A logarithmicly scaled plot with the size of $K_n$ ($y$), against the value of $n$ ($x$). Plotted are the points in Table \ref{tab:polycomplexities}, the lower bound in blue, the upper bound in green, and the conjectured upper bound in dotted black.}
    \label{fig:polyplot}
\end{figure}

Meanwhile, the notion of modified complexity in the polynomial ring offers a variety of avenues of exploration. One useful tool to explore them is to notate modified complexity expressions using a list as shorthand, partially inspired by our recursive algorithm for using computers to generate optimal expressions in a variety of other rings (see Appendix).

\begin{definition}
    Let the \textbf{shorthand list} $L$ of a modified complexity expression $f(x)$ be an ordered list consisting of 1's and $-1$'s and define it recursively as follows:\begin{itemize}
        \item $x = \{1\}$
        \item $f(x) + x$ is represented by appending 1 to the end of the list $L$ which represents $f(x)$
        \item $f(x) \cdot x$ is represented by appending -1 to the end of the list $L$ which represents $f(x)$
    \end{itemize}
\end{definition}
For example, 
\[
(((x+x) x)+x)x + x = \{1, 1, -1, 1, -1, 1\}.
\]

We can observe the following trivial properties of a polynomial and its shorthand list:

For a modified complexity expression $f(x)$ represented by $L$,
\begin{itemize}
    \item $\deg(f(x))$ is one more than the number of $-1$'s in $L$.
    \item The sum of the coefficients of $f(x)$ is the number of $1$'s in $L$, which is equal to $f(1)$.
    \item $|||f(x)||| = \text{len}(L)$
\end{itemize}

\begin{prop}\label{prop:degreeHorner}
    $$|||f(x)||| = \deg(f(x)) + f(1) -1$$
\end{prop}
\begin{proof}
    This follows from the properties listed above. 
\end{proof}

\begin{remark}
    Note that this is a special case of Lemma \ref{lem:deg}.
\end{remark}

\begin{prop}
    The number of unique $f(x)$'s that satisfies that $|||f(x)||| = k$ and $\deg(f(x)) = a$ for positive integers $k, a \ne 1$ is exactly $\binom{k-1}{a-1}$.
\end{prop}
\begin{proof}
    We know that for $\deg(f(x)) = a$, there must be $a-1$ occurrences of $-1$'s in $L$. We also know that for $|||f(x)||| = k$, the length of $L$ is $k$. Since the first element of $L$ is always 1, we are looking for the amount of ways we can arrange $a-1$ occurrences of $-1$'s and $k-(a-1)$ occurrences of 1's, which is $\binom{k-1}{a-1}$.
\end{proof}

\section{Integers Modulo \texorpdfstring{$m$}{m}}
We also looked into integer complexity in a finite ring $\mathbb{Z}_m$ for integral $m$. Note that we are no longer restricted to using 1 as our base. In this section, we shall denote the base as $s$. 

\begin{obs}
The representations with base $s$ span every element of $\mathbb{Z}/m\mathbb{Z}$ if and only if $s \in U_m$.
\end{obs}

\begin{obs}
    Different values of $s$ yield different complexities for the same number. There are some $s$'s that yield larger complexities than others. 
\end{obs}

We can no longer place bounds on the complexity of a certain element as it varies based on $s$. We define the following term to aid us in our investigation. 

\begin{definition}
    The \textbf{inefficiency} $E$ of an element, which will be used as a base, $s \in \mathbb{Z}/m\mathbb{Z}$ is defined as
    \[
    E(m,s) = \max(||0||, ||1||, ||2||, \ldots ||m-1||)
    \]
    In other words, the maximum complexity of an element of $\mathbb{Z}/m\mathbb{Z}$ with $x=s$.  
\end{definition}

Certain questions arise after investigating inefficiency. We will use these questions to guide our investigation of complexity in a finite ring.  
\begin{enumerate}
    \item Which values of $s$ are the least inefficient?
    \item Which values of $s$ are the most inefficient?
    \item Is there a formula to determine which $s$ are the least inefficient? Alternatively, what are the properties of the least inefficient $s$? 
    \item Is 0 always the most complex number? 
\end{enumerate}

Through computation, we know that 4 is false. One counterexample is with a base of 9 in $\mathbb{Z}/11\mathbb{Z}$. There is also the obvious counterexamples with a base of $-1$, since $(-1)(-1) + (-1) = 0$ is an expression of complexity 3.

\begin{conj} \label{conj:4.1}
    The inefficiencies have a logarithmic upper and lower bound, as appears to be true in Figure \ref{fig:AllInefficiencies}.
\end{conj}
After looking at some numerical examples, we also conjecture the following:

\begin{conj} \label{conj:inefficient1}
    For all $m$, when $s=1, E(m,s)$ is the highest out of all the possible $s$ values.
\end{conj}

This conjecture could be explained by the fact that multiplying by 1 does not yield new values, whereas other bases $s$ can be multiplied to obtain other numbers (thus exhausting all the integers in $\mathbb{Z}/m\mathbb{Z}$ earlier, reduing the inefficiency).

\begin{conj} \label{conj:4.3}
    $E(m,-1)$ is the second or tied third most inefficient for sufficiently large $m$. 
\end{conj}

A similar heuristic for Conjecture \ref{conj:inefficient1} applies here to provide an intuition as to why this is likely true.

\begin{conj} \label{conj:4.4}
    $E(m,2)$ is the third most inefficient for odd $m$. 
\end{conj}

The general idea here is that since $x = 2$ allows the two expressions $x^2$ and $2x$ to evaluate to the same value, we only have one element of complexity 2, which allows us to produce fewer values of certain greater complexities.

\begin{obs}
    Other than our conjectures above, inefficiency seems to be unrelated to the order of $s$ or how large $s$ is.  
\end{obs}

\begin{obs}
    There always seems to be a ``dominant'' inefficiency. As in, there is a value $n_m$ where there exists many $s$ values such that $E(m,s) = n_m$. There are certain $m$'s where this dominant value is much more dominant, and there are some values where the distribution of inefficiency is more equal. Furthermore, sometimes the dominant inefficiency seems to be the smallest possible inefficiency, but this not always the case. 
\end{obs}

\begin{obs}
    For varying $s$ in $\mathbb{Z}/m\mathbb{Z}$, we noticed that most $E(43, s)$ output around the same value, namely 7.
\end{obs}

\begin{obs}
    It seems like the $x$ such that $\mathbb{Z}/x\mathbb{Z}$ have a high proportion of bases with inefficiencies equal $n$, the least inefficiency, correspond to the lowest $x$-value in the set of all $x$ with an equal minimum inefficiency.
\end{obs}

\begin{figure}[H]
    \centering
    \begin{minipage}{.33\textwidth}
        \centering
        \captionsetup{justification=centering,margin=0.5cm}
        \includegraphics[scale=.33]{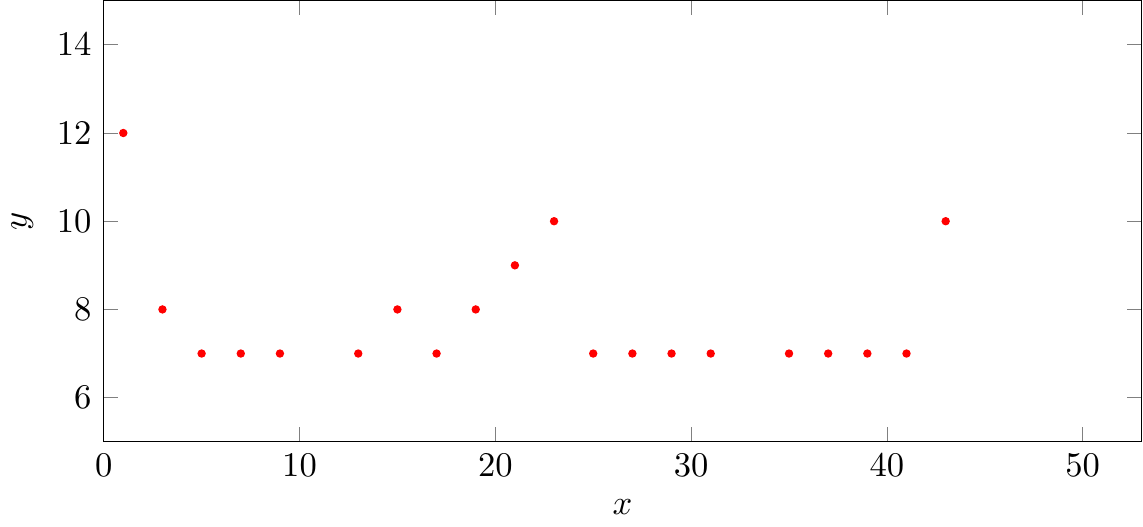}
        \caption{A plot of the inefficiencies ($y$) using a certain base ($x$) in $\mathbb{Z}_{44}$. In most cases, plotting $x$ and $y$ yields a plot similar to this one.}
        \label{fig:44}
    \end{minipage}%
    \begin{minipage}{.33\textwidth}
        \centering
        \captionsetup{justification=centering,margin=0.5cm}
        \includegraphics[scale=.33]{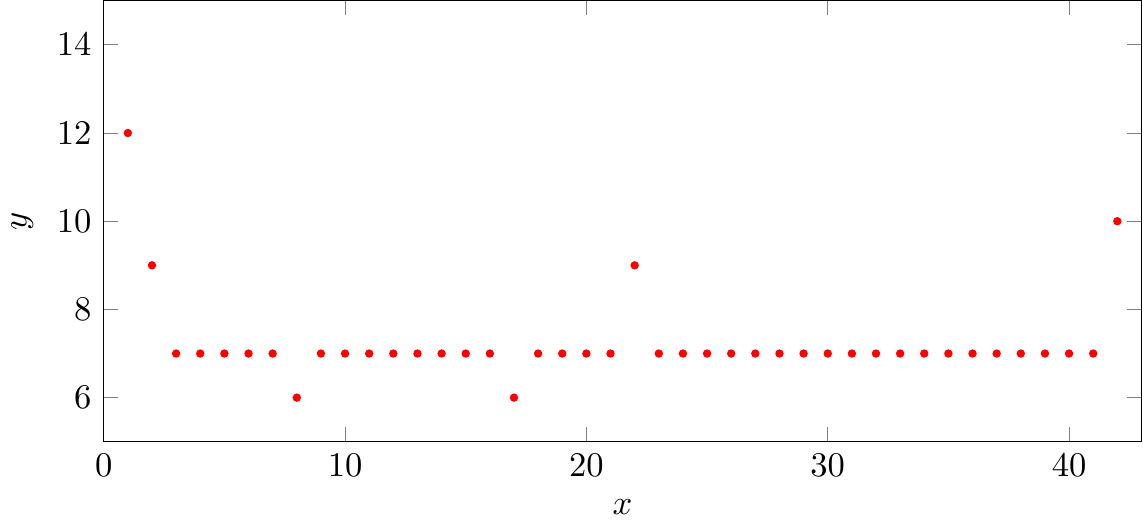}
        \caption{A plot of the inefficiencies ($y$) using a certain base ($x$) in $\mathbb{Z}_{43}$.\vspace{1.2 cm}}
        \label{fig:43}
    \end{minipage}
    \begin{minipage}{.33\textwidth}
        \centering
        \includegraphics[scale=.33]{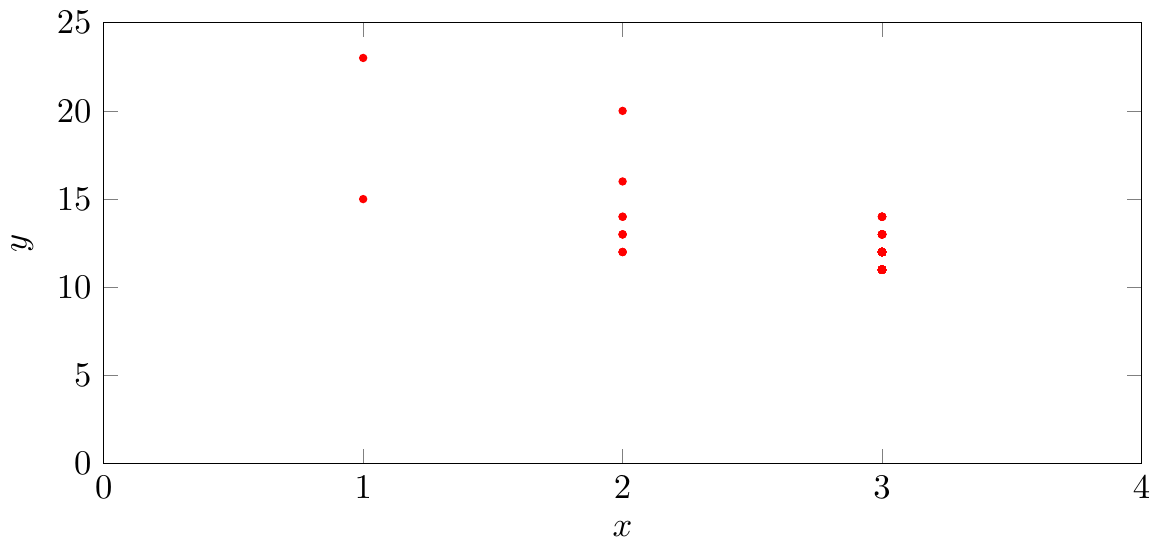}
        \captionsetup{justification=centering,margin=0.5cm}
        \caption{A plot of the resilience ($x$) vs. inefficiency ($y$) for $\mathbb{Z}/1009\mathbb{Z}$. \vspace{0.8cm}} 
        \label{fig:resilience}
    \end{minipage}
\end{figure}

For example, in Figure \ref{fig:44}, there are a lot of bases with an inefficiency of 7, and this corresponds to 44 being the first mod with no inefficiencies of 6. The intuition for this is that as the modulo increases, the inefficiencies of a given base tend to increase as there are more elements that need to be reached with combinations of that base. Thus at the point in which there are no more inefficiencies of $n-1$, there are a lot of other points which will have already increased up to $n$. Thus going up enough to `lose' the inefficiency of $n-1$ will result in more bases $s$ with inefficiency $n$.

We attempt to investigate the minimum inefficiency of a certain $\mathbb{Z}/m\mathbb{Z}$. 

\begin{figure}[H]
    \centering
    \includegraphics[width = .75\textwidth]{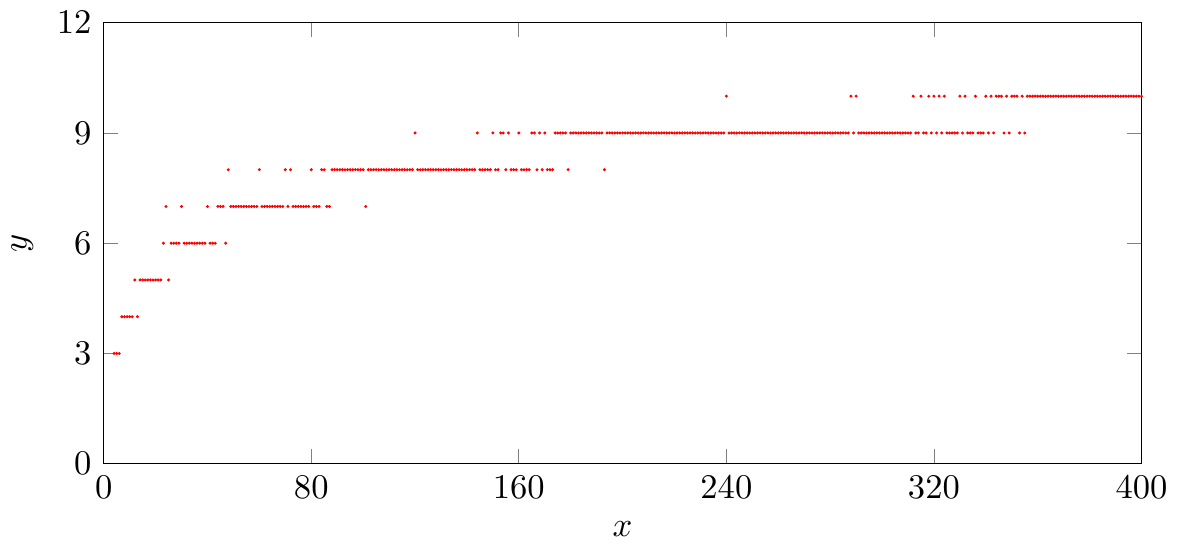}
    \caption{A plot of the minimum inefficiency ($y$) of $\mathbb{Z}/x\mathbb{Z}$.}
    \label{fig:mininefficiency}
\end{figure}

Let us consider the outliers that have a higher minimum inefficiency than others around it. Looking at Figure \ref{fig:mininefficiency}, we create the following data table:

\begin{center}
    \begin{tabular}{|c||c|c|c|c|c|c|c|c|c|c|c|c|}
        \hline 
        \multicolumn{13}{|c|}{$m$'s with low inefficiencies} \\
        \hline
        $m$ & 12 & 24 & 30 & 48 & 60 & 70 & 72 & 120 & 144 & 240 & 288 & 290\\
        \hline
        Inefficiency & 5 & 7 & 7 & 8 & 8 & 8 & 8 & 9 & 9 & 10 & 10 & 10\\
        \hline
\end{tabular}
\end{center}

The $m$'s are most of the time a multiple of 12, and generally seem to have many factors. Then, $\mathbb{Z}/m\mathbb{Z}$ would have many zero divisors. This could potentially explain why these numbers have such great inefficiencies: the abundance of factors results in them having less units than usual, so the ring has less 'tries' at having a unit base $s$ with a lower inefficiency.

\begin{definition}
    The \textbf{resilience} of $s$, (denoted $r(s)$) is the least $k$ such that for distinct modified complexity expressions $f(x), g(x)$ with complexity less than or equal to $k, f(s) = g(s)$. 
\end{definition}
For example, in $\mathbb{Z}_{10}, r(3) = 3$ because $3\cdot 3 = 3 + 3 + 3$ and this is the first instance where two expressions are equal when evaluated at 3. \\

\begin{prop}
    If $k$ is the resilience of $s$, then $s \geq k$.
\end{prop}
\begin{proof}
    We know that $s$ added to itself $s$ times, which has complexity $s$, is a repeat of $s^2$. Therefore, the latest possible occurrence of the first repeat has complexity $s$.

\end{proof}

\begin{obs}
    For all $m > 1$, in $\mathbb{Z}/m\mathbb{Z}, r(1) = r(2) = 2$. This occurs because the expressions $x$ and $x^2$ are equal whenever $x$ is 1, and $x^2 = 2x$ whenever $x$ is 2.
\end{obs}

In order to visually gain intuition about the importance of resilience, we can use graphs to represent modified complexity in $\mathbb{Z}/m\mathbb{Z}$:

\begin{definition} 
    A \textbf{modified complexity tree} with base $s$ (denoted by $\mct(s)$) is a binary tree with $s$ as the root. The left child of a node $t$ is $t+s$ and the right child is $t\cdot s$. For an example, see Figure \ref{fig:fftree}.
\end{definition}
\begin{figure} [H]
\centering
\begin{tikzpicture}[level distance=1cm,
  level 1/.style={sibling distance=4cm},
  level 2/.style={sibling distance=2cm},
  level 3/.style={sibling distance=1cm}]
  \node {3}
    child {node {6}
      child {node {2}
        child {node {5}}
        child {node {6}}
      }
      child {node {4}
        child {node {0}}
        child {node {5}}
      }
    }
    child {node {2}
        child {node {5}
            child {node {1}}
            child {node {1}}
        }
        child {node {6}
            child {node {2} }
            child {node {4} }
        }
    };
\end{tikzpicture}
        \captionsetup{justification=centering,margin=0.5cm}
\caption{Weak complexity tree for $\mathbb{Z}/7\mathbb{Z}$ and base $3$}
\label{fig:fftree}
\end{figure}
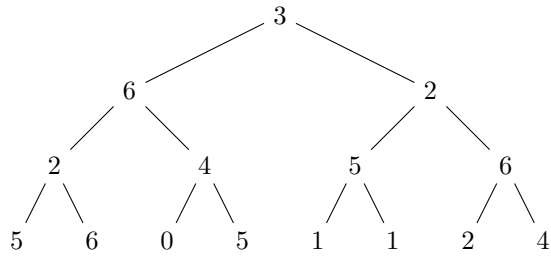

\begin{obs}
    It seems as though the more resilient a base is, the less inefficient it is. See \ref{fig:resilience}. This isn't always the case: for example, in many $\mathbb{Z}/x\mathbb{Z}$, there are elements with resilience 2 that are more inefficient than elements with resilience 3, but it generally seems to be ``likely'' that given two elements, the one with the higher resilience has the lower inefficiency.\\
    
    This follows from the intuition that one may get trying to make an MCT: having a repeat earlier gets rid of an entire potential branch of the tree, requiring potentially more levels of complexity to supplement for the loss.
\end{obs}

\begin{definition}
    A \textbf{modified complexity graph} with base $s$, (denoted by $\mcg(s)$), is a mct(s) without repeated nodes. For an example, see Figure \ref{fig:ffgraph}. 
\end{definition}

\begin{figure}[H]
    \centering
    \begin{tikzpicture}[main/.style = {draw,circle}, node distance =1.5 cm]
    \node[main] (3) {\small 3}; 
    \node[main] (4) [below of=3] {\small 4};
    \node[main] (0) [below right of=3] {\small 0}; 
    \node[main] (6) [below left of=3] {\small 6};
    \node[main] (5) [below of=4] {\small 5}; 
    \node[main] (2) [left of=6] {\small 2};
    \node[main] (1) [right of=0] {\small 1};
    \draw[->] (0) -- (3);
    \draw[->] (0) edge [loop right] (0);
    \draw[->] (1) edge [bend right = 50] (3);
    \draw[->] (3) -- (6);
    \draw[->] (6) -- (4);
    \draw[->] (1) edge [bend left = 10] (4);
    \draw[->] (4) -- (0);
    \draw[->] (6) -- (2);
    \draw[->] (4) -- (5);
    \draw[->] (2) -- (5);
    \draw[<->] (2) -- (6);
    \draw[->] (3) edge [bend right= 50] (2);
    \draw[->] (5) -- (1);
    \draw[->] (5) edge [bend right=20] (1);
    
    \end{tikzpicture} 
        \captionsetup{justification=centering,margin=0.5cm}
    \caption{Connected weak complexity graph with base 3 in $\mathbb{Z}/7\mathbb{Z}$.}
    \label{fig:ffgraph}
    \end{figure}
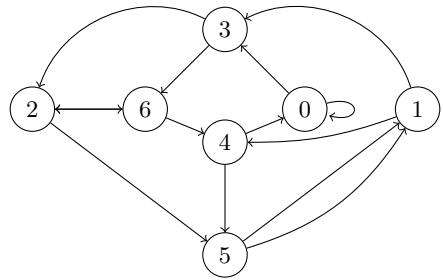

\begin{remark}
    Note that every node of this connected graph has two arrows leading out and two arrows leading in, since every node has an additive and multiplicative inverse.
\end{remark}

\begin{figure}[H]
    \centering
    \includegraphics[width=.8\textwidth]{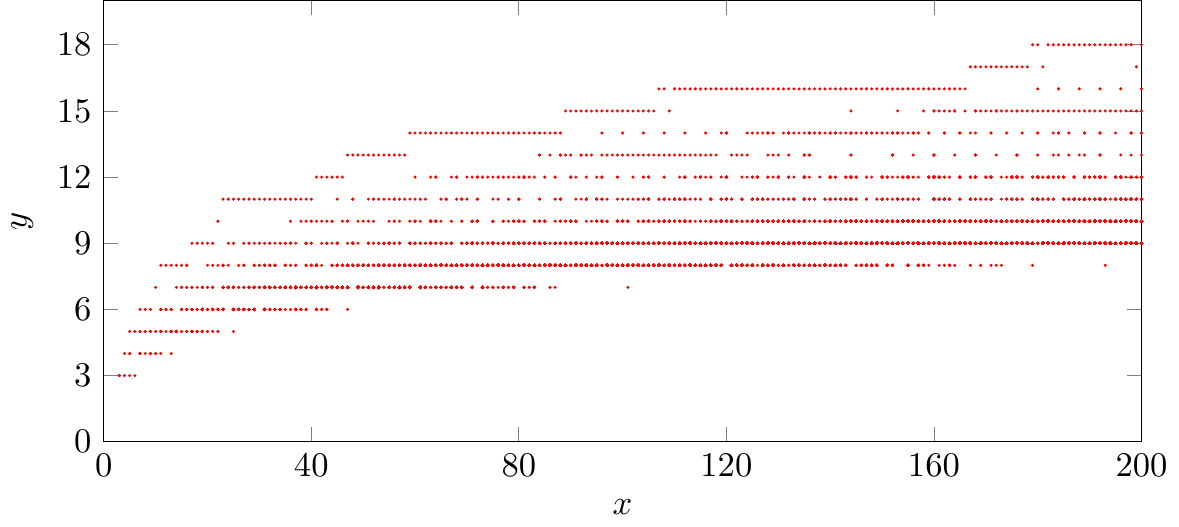}
        \captionsetup{justification=centering,margin=0.5cm}
    \caption{A plot of all of the inefficiencies ($y$) in $\mathbb{Z}/x\mathbb{Z}$.}
    \label{fig:AllInefficiencies}
\end{figure}

From here, we return to the question of what values of $s$ are most and least efficient. Observing Figure \ref{fig:AllInefficiencies}, we notice a distinct group of outliers that have higher inefficiencies than the rest. Through observation, we know that these have base $s=1$. At one point, we conjectured that perhaps this is a reflection of integer complexity in the naturals. To be precise, we once conjectured that for all $a \in \mathbb{Z}/m\mathbb{Z}$
\[
E(m,1) = \max(||a||). 
\]
Note that $||a||$ refers to integer complexity in the naturals. 

This would imply that inefficiency for $s = 1$ increments when $m$ is the least element of a complexity class. This is the explanation for the presence of mostly continuous line segments making up the inefficiencies for $s = 1$. 

However, there are counterexamples for this. One such counterexample is $107 \in \mathbb{Z}_{109}$. We computed that $||107||_1 = 15$, since $107 \equiv 216 \equiv 2^3\cdot 3^3 \pmod{109}$ but $||107|| = 16$.

Other than a small number of counterexamples, this conjecture generally seems to hold true. These counterexamples happen in $\mathbb{Z}/m\mathbb{Z}$ whenever for all naturals $a \leq m$, there is another integer $a'$ such that $a \equiv a' \mod m$ and $a'$ has an unusually low complexity.

In the case of $m = 109$, there were two factors that contributed to this being a counterexample. Firstly, 109 is barely larger than 107, the first element with complexity 16, so 107 was the only natural number up to 109 with a complexity of 16. Secondly, $107 \equiv 216 \mod 109$, and since 216 is ``very composite'' ($216 = 2^3\cdot3^3$), it happens to have a lower-than-ordinary complexity.\\

\begin{theorem} \label{thm:modupper}
For a base $s>1$, 
\[
   E(m,s) \leq \frac{s}{\ln s} \ln (m) + s - 1
\]
\end{theorem}
\begin{proof}
Since $E(m,s)$ is the largest complexity of an element $n$ in $\mathbb{Z}$ with base $s$, it is sufficient to prove an upper bound for $||n||_s$. In a similar fashion to our previous proofs, we want to express $n$ in base $s$ using Horner's method. However, we run into a problem if $s$ does not divide $n$ because we would get a remainder with unknown complexity.

Multiplying all elements of $\mathbb{Z}/m\mathbb{Z}$ by $s$ is a permutation since we already established that $s$ and $m$ are coprime. Therefore, the question is reduced to: what is the maximum number of $s$'s needed to express an element of the set $\{s, 2s, \ldots, (m-1)s\}$? Since all of these elements are divisible by $s$ in the naturals, we can use Horner's method to express them.

Suppose the base $s$ expansion of $n$ is:
\[
     n = s_1 + s(s_2 + s(s_3 + s(\ldots 
\]
Then, multiplying $s$ to both sides, 
\[
    ns = s \cdot s_1 + s(s \cdot s_2 + s(s \cdot s_3 + s(\ldots 
\]

The base $s$ representation of $n$ for some integer $n$ has $\lfloor\log_s{n}\rfloor + 1$ digits, in the naturals. Each digit $s_i$ is at most $s - 1$, so worst case scenario, each $s\cdot s_i$ takes at most $s - 1$ $s$'s. There are also $\lfloor\log_s{n}\rfloor$ $s$'s being multiplied.

Thus, adding these together, we get an upper bound of $(s-1)\left(\lfloor\log_s{n}\rfloor + 1\right)+\lfloor\log_s{n}\rfloor$. This is less than $s\log_s{n} + s - 1 = \frac{s}{\ln s} \ln (n) + s - 1$.

This is maximized when $n$ is set to $m$ (we don't have a way to obtain 0 with Horner's algorithm).
\end{proof}

\begin{remark}\label{rem:Ebound}
Note that this bound has a linear component, and for large $s$, this bound becomes very loose. It also falsely implies that inefficiency is correlated to the size of the base. We want to find a way to reduce $s$ if possible. This following corollary provides us with a powerful tool to reduce the upper bound in certain cases. 
\end{remark}

\begin{corollary}
For a base $s>1$,
\[
   ||n||_s \leq ||a||_s\left(\frac{a}{\ln a} \ln (m) + a - 1\right)
\]
\end{corollary}
\begin{proof}
    We can instead write the $n$ in base $a$; then this bound follows. 
\end{proof}

\begin{remark}
    We would want to choose an $a$ such that both $a$ and $||a||_s$ is small. Then, this would successfully reduce our bound. 
\end{remark}

\begin{corollary}
    $\frac{3}{\ln 3} \ln (m) + 2$ acts as an upper bound on the minimum inefficiency for any base $s$ in $\mathbb{Z}/m\mathbb{Z}$.
\end{corollary}

\begin{proof}
    This can be seen by simply plugging in 3 for the value of $s$ in the bound described in Theorem \ref{thm:modupper}. 3 is the positive integral value that causes $\frac{s}{\ln s}$ to have its lowest value.
\end{proof}

\begin{remark}\label{rmk:boundgraph}
    This proof can also be seen geometrically. We can consider a mc-graph in which all the multiples of $s$ are arranged in a circle for a $\mathbb{Z}/p\mathbb{Z}, p$ a prime. For an example, see Figure \ref{fig:circle}. Then we can index each multiple $as$ as $a$. 
    
    Then the operations of modified complexity are directly analogous to the operations in the base representation (addition of $s$ is the addition of $1$ on the graph, shifting one node counterclockwise; multiplication by $s$ multiplies the index by $s$). 

    As a result, the index of our base $s$ is 1, and these operations make expressing any node of the graph in terms of $s$ equivalent to expressing their index in terms of 1 using Horner's rule.
\end{remark}

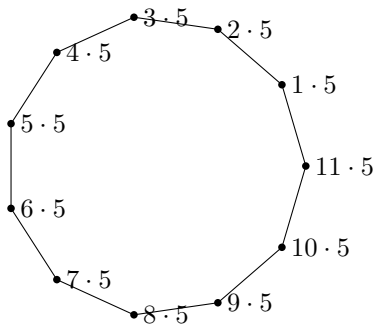
\begin{figure}[H]
    \centering
    \begin{tikzpicture}
    \equic[2cm]{11}
    \end{tikzpicture}
    \caption{A subgraph (excluding multiplication by $s$) of an mc-graph for $\mathbb{Z}/11\mathbb{Z}$ with $s = 5$}
    \label{fig:circle}
\end{figure}

\begin{obs}
    From trying some examples, it appears that every mc-graph has a cycle of length $r$ for all $r$, a natural up to $x$, the modulo.
\end{obs}

We are interested in strengthening the upper bound on the lowest possible inefficiency for a given field (not a ring, as will be clarified below) $\mathbb{Z}/m\mathbb{Z}$ from our already established bound of $\frac{3}{\ln 3}\ln{m}+2.$ Although we were unsuccessful in this endeavor so far, we shall outline our general method for attempting this.

The general idea is to guarantee the existence of a certain number of distinct elements with complexity under a certain value (for an existing value of $s$).

We shall go about this by trying to find the base $s$ with the maximum ``resilience'' in terms of modified complexity. If this resilience is equal to $r$, then there are $2^{r-1}-1$ distinct elements with modified complexity under $r$. Since modified complexity is an overestimate of actual complexity, these $2^{r-1} - 1$ elements all have complexity under $r$.

We will ``determine'' what base $s$ to choose by means of elimination. For example, we know that if $s = 1$, then the polynomials $x$ and $x^2$ will repeat each other, so we won't let $s$ be 1. Similarly, $s = 2 \Rightarrow x^2 = 2x$, so $s \neq 2$. Generally, we shall pair and set equal polynomials $f(x), g(x).$ $f(x) - g(x)$ has $\deg(f)$ roots if, WLOG, we let $f$ have the higher degree (importantly, this is only true if there are no zero-divisors in $\mathbb{Z}/m\mathbb{Z}$, so we are only considering fields); both polynomials are divisible by $x$ since they are expressions, so we will find a maximum of $\deg(f) - 1$ new roots $s$ to be thrown out. This process will be repeated until we have thrown out more than $p$ roots.

Note that this strategy does end up under-counting $r$ severely since it does not account for repeated roots whatsoever. However, this does not impact the existence of those undercounted elements with complexity under $r$.

\begin{lemma}\label{prop:lowerresilience}
    The highest resilience in modified complexity among those of all bases in $\mathbb{Z}/p\mathbb{Z}$ is greater than or equal to the largest $k$ such that $\left(k-2\right)4^{k}-\left(k-4\right)2^{k} \le 2p$.
\end{lemma}
\begin{proof}
    First, we begin by showing that this expression
    \begin{equation*}\label{eqn:maxresiliencesum}
        \sum_{D_1 = 1}^{k} (D_1 - 1)\sum_{D_2=1}^{D_1} \sum_{k_1 = D_1}^{k} \binom{k_1-1}{D_1 - 1} \sum_{k_2 = D_2}^{k} \binom{k_2 - 1}{D_2 - 1}
    \end{equation*}
    is an upper bound for the number of integers that the base can't be if the modified resilience of the base is greater than $k$. In other words, within $k$ complexities, with the strategy detailed above, we will have eliminated this many possibilities for the value of $s$. 

    Suppose that we are setting equal two polynomials $f$ and $g$ with complexity up to $k$ and eliminating all of the resulting roots (other than 0). Then let $D_1, D_2$ be the respective degrees of $f$ and $g$ and $k_1, k_2$ be their respective modified complexities. 
    
    WLOG, suppose $D_1 \geq D_2$. Then for each pair $f$ and $g$, we would need to eliminate $D_1$ roots for each pair.

    Next given specific values of $D_1, D_2, k_1, k_2$, there are $\binom{k_1-1}{D_1 - 1}\cdot \binom{k_2 - 1}{D_2 - 1}$ polynomials $f$ and $g$ with those degrees and modified complexities.

    The degree $D_1$ can range anywhere from 1 to $k$, whereas $D_2$ can be anything from 1 to $D_1$. The modified complexities $k_1$ and $k_2$ can go up to $k$ but must be greater than the degrees of the respective polynomials associated with them. Therefore, in total, the number of possible values of $s$ we are eliminating is:

    \[
        \sum_{D_1 = 1}^{k} \sum_{D_2=1}^{D_1} \sum_{k_1 = D_1}^{k}  \sum_{k_2 = D_2}^{k} \binom{k_1-1}{D_1 - 1}\cdot \binom{k_2 - 1}{D_2 - 1}\cdot (D_1 - 1),
    \]

    which can be rewritten as

    \[
        \sum_{D_1 = 1}^{k} (D_1 - 1)\sum_{D_2=1}^{D_1} \sum_{k_1 = D_1}^{k} \binom{k_1-1}{D_1 - 1} \sum_{k_2 = D_2}^{k} \binom{k_2 - 1}{D_2 - 1}.
    \]
    
    Then, we can show that this is less than or equal to $(2^k - 1)(2^{2k-1} - 2^k + 1)$ with some algebra:
    \begin{align*}
        \sum_{D_1 = 1}^{k}(D_1 - 1)\left( \sum_{D_2=1}^{D_1} \left(\sum_{k_1 = D_1}^{k} \binom{k_1-1}{D_1 - 1} \left(\sum_{k_2 = D_2}^{k} \binom{k_2 - 1}{D_2 - 1}\right)\right)\right)\\
        = \sum_{D_1 = 1}^{k}  (D_1 - 1)\left(\sum_{k_1 = D_1}^{k} \binom{k_1-1}{D_1 - 1} \left(\sum_{D_2=1}^{D_1}\left(\sum_{k_2 = D_2}^{k} \binom{k_2 - 1}{D_2 - 1}\right)\right)\right)
    \end{align*}
    Then let $k_1 - 1 = a, D_1-1 = b, k_2 - 1 = c, D_2 - 1 =d$ to get:
    \begin{equation}\label{eqn:commutedsum}
        \sum_{b = 0}^{k-1} b\left(\sum_{a = b}^{k-1} \binom{a}{b} \left(\sum_{d=0}^{b}\left(\sum_{c = d}^{k-1} \binom{c}{d}\right)\right)\right)
    \end{equation}
    Then note that 
    
    \[\sum_{d=0}^{b}\sum_{c = d}^{k-1} \binom{c}{d} =
    \begin{tabular}{P{0.55cm}P{0.55cm}P{0.55cm}P{0.55cm}P{0.55cm}P{0.55cm}P{0.55cm}P{0.55cm}}
        $\binom{0}{0}$ & + & $\binom{1}{0}$ & + & $\binom{2}{0}$ & $\cdots$ & + & $\binom{k-1}{0}$ \\
                      &+ &$\binom{1}{1}$ & + &$\binom{2}{1}$ & $\cdots$ & + & $\binom{k-1}{1}$ \\
                     &       &       & + &$\binom{2}{2}$ &  $\cdots$& + & $\binom{k-1}{2}$ \\
                     &     & &   &      &              & $\ddots$ &  $\vdots$      \\
                     &              &     &        & $\binom{b}{b}$ & $\cdots$ & + &  $\binom{k-1}{b}$      
    \end{tabular}
    \]
    
    This then is less than $0 + 2^1 + 2^2 + \cdots + 2^{k-1}$. Using the geometric sum formula, this equals $2^{k} - 1$. Thus Equation \ref{eqn:commutedsum} is $\le (2^{k} - 1)\sum_{b = 0}^{k-1} b\sum_{a = b}^{k-1} \binom{a}{b} $ (distributive law). 
    
    Then see that 
    
    \[\sum_{b = 0}^{k-1} b\sum_{a = b}^{k-1} \binom{a}{b} =
    \begin{tabular}{P{0.55cm}P{0.55cm}P{0.55cm}P{0.55cm}P{0.55cm}P{0.55cm}p{2cm}}
                $\binom{1}{1}$ & + &$1\binom{2}{1}$ & + & $\cdots$ & + & $1\binom{k-1}{1}$ \\
                             & + &$2\binom{2}{2}$ & + & $\cdots$ & + & $2\binom{k-1}{2}$ \\
                             &   & & &          & $\ddots$ &  $\vdots$      \\
                             &   & & &           & &  $(k-1)\binom{k-1}{k-1}$      
    \end{tabular}
    \]

    Consider an arbitrary $i$th column of this triangle, the sum of which evaluates to:

    \[
    C_i = \binom{i}{1} + 2\binom{i}{2} + \cdots + (i-1)\binom{i}{i-1} + i\binom{i}{i}
    \]
    Then:
    \begin{align*}
    2C_i = &i\binom{i}{i} + \left(\binom{i}{1} + (i-1)\binom{i}{i-1} \right) + \left(2\binom{i}{2} + (i-2)\binom{i}{i-2} \right) + \cdots \\
    &+ \left((i-1)\binom{i}{i-1} + \binom{i}{1}\right)+ i\binom{i}{i}\\
    = &i\binom{i}{0} + \left(\binom{i}{1} + (i-1)\binom{i}{1} \right)+ \left(2\binom{i}{2} + (i-2)\binom{i}{2} \right) + \cdots \\
    &+ \left((i-1)\binom{i}{i-1} + \binom{i}{i-1}\right)+ i\binom{i}{i}\\
    = &i\left( \binom{i}{0} + \binom{i}{1} + \cdots + \binom{i}{i} \right) \\
    = & i\cdot 2^i\\
    C_i = &i\cdot 2^{i-1}
    \end{align*}
    By adding all the columns together, we get:
    \begin{align*}
         \sum_{b = 0}^{k-1} b\sum_{a = b}^{k-1} \binom{a}{b} &= 1\cdot 2^0 + 2 \cdot 2^1 + \cdots (k-1)\cdot 2^{k-2}\\
         &= \sum_{i=1}^{k-1}\left(  \sum_{j=i-1}^{k-2} 2^j \right)\\
         &= \sum_{i=1}^{k-1}\left(  2^{k-1}-2^{i-1} \right)\\
         &= \left( (k-1)\cdot 2^{k-1} - \left( 2^{k-1} - 1 \right)\right)\\
         &= \left( (k-2)\cdot 2^{k-1} +1\right)
    \end{align*}
    by the geometric sequence sum formula.

    Therefore, plugging back in, we get an upper bound of $(2^k - 1)\left( (k-2)\cdot 2^{k-1} +1\right)$.
    
    Since we are trying to eliminate as many bases as possible until we run out, we need our value of $r$ to be the greatest value of $k$ such that 
    \[
    (2^k - 1)\left( (k-2)\cdot 2^{k-1} +1\right) \le p - 1.
    \]

    This can be simplified to state that we need our value of $r$ to be the greatest value of $k$ such that 
    \[
    (k-2)2^{2k-1} + 2^k - (k-2)2^{k-1} \le p.
    \] 
    This further simplifies to
    \[
    \left(k-2\right)4^{k}-\left(k-4\right)2^{k} \le 2p. 
    \]
\end{proof}

\begin{remark}
    It may be possible to improve this upper bound by noting that in the shorthand lists for two expressions, if the ending numbers are the same, after setting the polynomials equal to each other, undoing that last operation would reduce the equality to an equality of two simpler polynomials. For example, if we have already found the values of $x$ such that $2x^3 = x^2 + x$, we need not check the values for which $2x^3 + x = x^2 + 2x$ or $2x^4 = x^3 + x^2$ since these equations could be reduced to the former. This means that we could drastically cut down the number of starting values $s$ needing to be thrown out; however, doing so makes the counting more difficult, so we have not done it above.
\end{remark}

From here, we can further simplify (while also loosening) the bound to an exponential one: observing graphs of these functions reveals to us that for all integral values, we saw that $\left(k-2\right)4^{k}-\left(k-4\right)2^{k} < b^x$ for $b = (4^7-2^7)^\frac{1}{6}$ (note that these two functions are equal at $k = 6$, and choosing a smaller value of $b$ will cause the inequality not to hold). The following proof will arrive at this result algebraically.

\begin{lemma}
    The highest resilience in modified complexity among those of all bases in $\mathbb{Z}/p\mathbb{Z}$ is greater than or equal to $\lfloor\log_{b} (2p) \rfloor$, for $b = (4^7-2^7)^\frac{1}{6}$. 
\end{lemma}

\begin{proof}
    Let $f(k) = \left(k-2\right)4^{k}-\left(k-4\right)2^{k}$. We shall start by showing that $f(k) < b^k$ for all natural values of $k$. This will be proved by induction. As a base case, this can be shown to hold true computationally for 1, 2, 3, 4, 5, 6, and 7.

    Next, consider the function: 
    \[
    \frac{f(k+1)}{f(k)} = \frac{(k-1)4^{k+1}-(k-3)2^{k+1}}{(k-2)4^{k}-(k-4)2^{k}} = \frac{4(k-1)4^{k}-2(k-3)2^{k}}{(k-2)4^{k}-(k-4)2^{k}} = \frac{4(k-1)-\frac{2(k-3)}{2^{k}}}{(k-2)-\frac{(k-4)}{2^{k}}}.
    \]

    For $k \geq 3$, this is less or equal to $\frac{4(k-1)}{(k-2)-\frac{(k-4)}{2^{k}}}$. Furthermore, for $k \geq 6$, the denominator is greater than or equal to $(k-2)-\frac{2}{2^{6}}$, so the function will be less than $\frac{4(k-1)}{(k-2)-\frac{1}{32}}$.

    Thus, since $7 > 6> 3$, for $k \geq 7, \frac{f(k+1)}{f(k)} < \frac{128(k-1)}{32(k-2)-1} $. Furthermore, it can be seen by computation that $\frac{128(k-1)}{32(k-2)-1} < b$ at $k=7$. Since $\frac{128(k-1)}{32(k-2)-1}$ is a hyperbola and can easily be seen to be decreasing after 7, we can also conclude that $\frac{128(k-1)}{32(k-2)-1} < b$ for all $k \geq 7$. Then, by the transitive property, we get that $\frac{f(k+1)}{f(k)} < b$ for all $k \geq 7$.

    Now, assume that $f(k) < b^k$ for some natural number $k$ greater than 7. Since $\frac{f(k+1)}{f(k)} < b$, we know $f(k+1) < f(k)\cdot b$, and using our inductive assumption, we get that $f(k+1) < b^{k+1}.$

    This means that by the transitive property, $b^k\cdot \frac{1}{2} - 1$ is also an overestimate for the number of possible values of $s$ we eliminate within the first $k$ modified complexities. We can keep eliminating values of $s$ until we run out and have eliminated more than $p - 1$ values of $s$, so the maximum modified resilience is largest value of $k$ such that $b^k\cdot \frac{1}{2} - 1 \leq p - 1$. Therefore, simplifying, we need the largest value of $k$ such that $k \leq \log_b (2p)$, which is just $\lfloor\log_{b} (2p) \rfloor + 1$.
\end{proof}

Essentially, what we have stated is that there exists a value $s \in \mathbb{Z}/p\mathbb{Z}$ such that all polynomial expressions of the first $\lfloor\log_{b} (2p) \rfloor$ modified complexities are distinct when evaluated at $s$. This means that for each complexity $i < \lfloor\log_{b} (2p) \rfloor + 1$, we have $2^{i-1}$ elements with that modified complexity (whose real complexities could be even less). This gives us $2^{\lfloor\log_{b} (2p) \rfloor} - 1$ elements in total, which turns out to be a root expression of $p$ between $p^\frac{1}{2}$ and $p^\frac{1}{3}$. We have not succeeded in doing this yet, but the task of finding a lower bound may be accomplished by finding a way to optimally combine all of these elements, for which an upper bound on complexity is known, to generate any element.\\

\section{Directions for Future Work}
In the course of this paper, we have opened new cans of worms regarding integer complexity. Let us close our paper with the following remarks and questions.  
\subsection{Cyclotomic Fields}
Working in cyclotomic fields has yielded us some surprising results. We still have many unresolved questions. 
\begin{itemize}
    \item Our conjectures: \ref{conj:onebestwecando}, \ref{conj:whenisoneone}, \ref{conj:minimality}, \ref{conj:bestTheoreticalUpper}, and \ref{conj:powerliftingBound} 
    \item How tight are our bounds for larger $k$ values? (Due to computational restrictions, we have not been able to compute significant amounts of data.)
    \item Is there a bound on cyclotomic complexity that is tighter than ours for sufficiently large $n$? 
    \item The set of minimal numbers $\mathbb{M}$ is a promising set and more investigation should be done. 
    \item What are the relationships between the different $k$ values as its base? Some interesting things to explore include looking at the size of the complexity classes and computatability of complexities for ``nicer'' $k$, such as $k=2,4,8,$ etc.
    \item (These problems were given to us by Zelinsky as guiding problems): Do any of the following exist, and if so what is the value?
\begin{align*}
    \lim_{k\rightarrow \infty} \frac{||n||_k}{k} \\
    \lim_{k\rightarrow \infty} \frac{||n||_{p^k}}{p^k}\\
    \lim_{k\rightarrow\infty}     \frac{||n||_{k!}}{k!}  
\end{align*}

    Note that for $n=1$, we have done the first two limits, by Theorem \ref{theorem:primepowers1comp} (first doesn't, second equals 1).
    \item We noticed that we can use minimal polynomials in cyclotomic complexity (see Subsection \ref{subsect:polynomialrings}). Perhaps we can apply more advanced properties of minimal polynomials to aid research in cyclotomic complexity. 
    \item The connection between Galois theory and integer complexity in cyclotomic rings is interesting in that it produced a proof that $||1||_p = p$ (Proposition \ref{prop:comp1}). But if there was a generalization of this proof to prime powers (the obvious approach since they have a generator and Theorem \ref{theorem:primepowers1comp} proves it) there should be analogous parallels in the proof of Theorem \ref{theorem:primepowers1comp} and the generalization, as was the case for Proposition \ref{prop:comp1} in Footnote \ref{footnote:galoisanalogues}. But the integers mod powers of $2$ don't have generators (except for 2,4), whereas a Galois theory proof for Theorem \ref{theorem:primepowers1comp} might suggest they have generators. Furthermore, Corollary \ref{corollary:1compPrimeBound} raises the possibility that $||1||_{2^a p^b}$ might equal $k$.
\end{itemize}

\subsection{Polynomial Rings} \label{subsect:polynomialrings}
\begin{itemize}
    \item Our conjectures: \ref{conj:2lowerbound}, \ref{conj:eUpperbound}
    \item We have applied complexity classes $K_k$ to $\mathbb{Z}/m\mathbb{Z}$. Can we apply this concepts elsewhere to prove other conjectures and better bounds? 
    \item Can we find better bounds with modified complexity? Can we use these results and generalize them to normal polynomial complexity? 
    \item What results can we transfer from the polynomial ring to the naturals and using a base in general? 
    \item Can we find results pertaining to our questions in finite polynomial rings?
    \item What other results can be obtained from analyzing the minimal polynomials and properties of every expression having coefficients in $\mathbb{N}_0$?
    \item Furthermore, what structure could be used if we expanded the polynomials to allow coefficients in Cyclotomic Rings? This is especially related to Conjecture \ref{conj:unitcomplexity}, but may also reveal insights in power lifting like in Conjecture \ref{conj:powerliftingBound}.
    \item It would also be interesting to consider the relation between optimal expression of a natural $n$, say $xf$, using the solutions of $xf =n$ as the base.
    \item It may be fruitful to search for/classify all functions $f, h$ and a sequence of polynomials $g_0, g_1, \ldots$ such that $g_{n+1} = f(g_n)$ and $||g_{n+1}|| = h(||g_n||)$.
    This may provide deeper insight as to why using powers of $3x$ in the Cyclotomic Rings was so effective, how to generalize this, and a deeper understanding of why it has to be true. Finding some may also provide a proof and measure for computational ``tricks'' and their effectiveness. For instance, multiplying by 10 has a ``trick'' of doing a bitshift. This is represented by $f(g) = gx$ and $h(x) = 1+x$ having this condition. Other tricks like multiplying by 11 are represented by $f(g) = xg + g$ (this doesn't satisfy this for lack of an $h$). Note that these tricks are base agnostic, like $x$ being any natural we want. Then if $f,h$ don't work, can we define an analogue to the deficit and bound the difference between the real complexity and the predicted complexity (from a $h$)? Thereby a measure for the effectiveness of computational ``tricks'' can be made.
    \item Are we able to generalize the more efficient algorithms in the current literature, such as in \cite{Cordwell2017Jun}, to the Cyclotomic Rings?
    \item If so, such an algorithm and an exact condition for Conjecture \ref{conj:unitcomplexity} (perhaps the condition is that $k = 2^a p$ and $2^{a-2} < p$ as suggestive by Corollary \ref{corollary:1compPrimeBound}) being true could provide a test for a natural being a prime power or 2 times a prime power.
    \item Considering the minimality conjecture \ref{conj:whenisoneone} in the light of the technique in Theorem \ref{theorem:primepowers1comp} may also lead to some progress.
\end{itemize}

\subsection{Integers Modulo \texorpdfstring{$m$}{m}}
\begin{itemize}
    \item Our conjectures: \ref{conj:4.1}, \ref{conj:inefficient1}, \ref{conj:4.3}, \ref{conj:4.4}
    \item Is there a way to determine the least inefficient $s$ in $\mathbb{Z}/m\mathbb{Z}$? 
    \item Can we generalize the techniques using in Theorem \ref{theorem:primepowers1comp} to be in finite fields in which a $k$-th root of unity exists?
    \item It may be possible to do further research through applying graph theory on a mc-graph to learn more about the structure of integer complexity in a finite field. Furthermore, this is a subset of the vehicle routing problem in which the directed edges are one-way equal length streets and we wish to find an optimal route between a home position and a customer. Analyzing further examples of mc-graphs with the complexity of a node replacing it may also reveal more insights.
    \item Another possible direction is to consider mc-graphs in the non-modified complexity. However, this loses the nice property of having every node having two edges going in and two edges going out of it.
    \item We could also try to approximate the growth of complexity classes with probabilities using something like the logistic growth model.
\end{itemize}

\section{Acknowledgements}
We would like to thank our mentor Joshua Zelinsky for acquainting us with the project and showing us resources that help with our project. We would also like to thank our counselor Aaryan Sukhadia for listening to and verifying our proofs and giving insightful comments. We thank John Sim and David Fried for running research labs, and also the Clay Mathematical Institute for cosponsoring this project. 

\printbibliography

\section{Appendix}

\subsection{Algorithm}

We created various programs to compute complexities of elements for us. The general approach can be described as follows:
\begin{enumerate}
    \item Let $x$ be associated with the empty ordered list $()$.
    \item Store all of the expressions that we have in a set
    \item Create a new expression by adding or multiplying two other expressions and check that this expression is not already in our set of expressions. Let us define the ordered list associated with this new expression as follows: \begin{enumerate}
        \item Concatenate (expression of the first ordered list, n, expression of the second ordered list)
        \item Let n be the length of this list. Note that $n$ is one less than the complexity since the length of the list is one less than the complexity.
        \item If we are multiplying, then make n negative
    \end{enumerate}
    \item Add this new expression to our set of expressions
\end{enumerate}

Here is the algorithm in pseudocode.
\begin{lstlisting}[language = Python]
k = #maximum complexity we are looking for 
#we wish to associate an optimal expression with an ordered list. 
#The length of the ordered list is the complexity of the expression. 
expressions = {[x, ()]}
for i in range k:
    for a in range i:
        for all element1s in expressions with length a:
                for all element2s in expressions with length i-a:
                    #this one tests addition
                    newexpression = element1[0] + element2[0]
                    if newexpression not in expressions:
                        newlist = element1[1] + (length(element1[1]) +
                        length(element2[1]) + 1) + element2[1] 
                        #this addition is list concatenation
                        append [newexpression, newlist] to expressions
                    #this one tests multiplication
                    newexpression = element1[0] * element2[0]
                    if newexpression not in expressions:
                        newlist = element1[1] - (length(element1[1]) +
                        length(element2[1]) + 1) + element2[1] 
                        append [newexpression, newlist] to expressions
                    
\end{lstlisting}

For example, $1,2,-6,1,-2,1,9,1,-2$ correponds to $((x + x + x) \times ((x + x) \times (x+x)))((x + x) \times x)$.

Here is a link to the data and programs used in this paper: \url{https://gitlab.com/partridgetrai\\toruntangled/integer-complexity-generalizations-in-various-rings/-/tree/main}.
\end{document}